\documentclass{amsart}

\usepackage{version}
\includeversion{arxiv}

\usepackage{amssymb,amsfonts,amsxtra}
\usepackage[all]{xypic}
\xyoption{dvips}
\usepackage{url}
\usepackage[colorlinks,plainpages,backref]{hyperref}
\usepackage{paralist}

\theoremstyle{definition}
\newtheorem{ntn}{Notation}[section]
\newtheorem{dfn}[ntn]{Definition}
\theoremstyle{plain}
\newtheorem{lem}[ntn]{Lemma}
\newtheorem{prp}[ntn]{Proposition}
\newtheorem{thm}[ntn]{Theorem}
\newtheorem{cor}[ntn]{Corollary}
\newtheorem{cnj}[ntn]{Conjecture}

\theoremstyle{remark}

\newtheorem{rmk}[ntn]{Remark}
\newtheorem{exa}[ntn]{Example}

\renewcommand{\AA}{\mathbb{A}}
\newcommand{\GG}{\mathbb{G}}
\newcommand{\PP}{\mathbb{P}}
\newcommand{\KK}{\mathbb{C}}
\newcommand{\NN}{\mathbb{N}}
\newcommand{\QQ}{\mathbb{Q}}
\newcommand{\RR}{\mathbb{R}}
\newcommand{\ZZ}{\mathbb{Z}}
\newcommand{\cO}{\mathcal{O}}
\renewcommand{\cH}{\mathcal{H}}
\renewcommand{\cL}{\mathcal{L}}
\newcommand{\del}{\partial}
\newcommand{\fg}{\mathfrak{g}}
\newcommand{\fh}{\mathfrak{h}}
\newcommand{\fl}{\mathfrak{l}}
\newcommand{\fm}{\mathfrak{m}}
\newcommand{\fz}{\mathfrak{z}}
\newcommand{\gl}{\mathfrak{gl}}

\newcommand{\ideal}[1]{{\langle#1\rangle}}
\newcommand{\into}{\hookrightarrow}
\newcommand{\ol}{\overline}

\newcommand{\red}{\mathrm{red}}
\newcommand{\xymat}{\SelectTips{cm}{}\xymatrix}

\DeclareMathOperator{\Ad}{Ad}
\DeclareMathOperator{\Aut}{Aut}
\DeclareMathOperator{\coker}{coker}
\DeclareMathOperator{\defect}{def}
\DeclareMathOperator{\Der}{Der}
\DeclareMathOperator{\ddim}{\underline{dim}}
\DeclareMathOperator{\Ext}{Ext}
\DeclareMathOperator{\GL}{GL}
\DeclareMathOperator{\Gr}{Gr}
\DeclareMathOperator{\Hom}{Hom}
\DeclareMathOperator{\End}{End}
\DeclareMathOperator{\id}{id}
\DeclareMathOperator{\img}{im}
\DeclareMathOperator{\Mat}{Mat}
\DeclareMathOperator{\PGL}{PGL}
\DeclareMathOperator{\Rep}{Rep}
\DeclareMathOperator{\rk}{rk}
\DeclareMathOperator{\sign}{sign}
\DeclareMathOperator{\Sing}{Sing}
\DeclareMathOperator{\SI}{SI}
\DeclareMathOperator{\SL}{SL}

\DeclareMathOperator{\X}{X}
\DeclareMathOperator{\Z}{Z}

\numberwithin{equation}{section}

\begin{document}

\title{Free divisors in prehomogeneous vector spaces}

\author{Michel Granger}
\address{
M. Granger\\
Departement de Math\'ematiques\\
Universit\'e d'Angers\\
2 Bd Lavoisier\\
49045 Angers\\
France
}
\email{granger@univ-angers.fr}

\author{David Mond}
\address{
D. Mond\\
Mathematics Institute\\
University of Warwick\\
Coventry CV47AL\\
England
}
\email{D.M.Q.Mond@warwick.ac.uk}

\author{Mathias Schulze}
\address{
M. Schulze\\
Department of Mathematics\\
Oklahoma State University\\
Stillwater, OK 74078\\
United States}
\email{mschulze@math.okstate.edu}

\date{\today}



\begin{abstract}
We study linear free divisors, that is, free divisors arising as discriminants in prehomogeneous vector spaces, and in particular in quiver representation spaces.  
We give a characterization of the prehomogeneous vector spaces containing such linear free divisors.
For reductive linear free divisors, we prove that the numbers of geometric and representation theoretic irreducible components coincide.
As a consequence, we find that a quiver can only give rise to a linear free divisor if it has no (oriented or unoriented) cycles.
We also deduce that the linear free divisors which appear in Sato and
Kimura's list of irreducible prehomogeneous vector spaces are the only irreducible reductive linear free divisors.

Furthermore, we show that all quiver linear free divisors are strongly Euler homogeneous, that they are locally weakly quasihomogeneous at points whose corresponding representation is not regular, and that all tame quiver linear free divisors are locally weakly quasihomogeneous.
In particular, the latter satisfy the logarithmic comparison theorem.
\end{abstract}

\maketitle
\tableofcontents

\section{Introduction}

A reduced hypersurface $D\subset\AA^n_\KK=V$ is called a linear free divisor if the sheaf of logarithmic vector fields admits a global basis of linear vector fields.
It turns out that $D$ is then the discriminant in a prehomogeneous vector space $(G,\rho,V)$ with $\dim G=n$ defined by a homogeneous polynomial $f\in\KK[V]$ of degree $\deg(f)=n$. 
In this sense, this class of divisors is opposite to the class of central hyperplane arrangements, where the irreducible components are linear but the basis of logarithmic vector fields is in general not.
(Only the normal crossing divisor is a member of both classes.) 
Recently linear free divisors have been subject to intensive research \cite{BM06,GMNS09,GMS09,Sev09}.

The main two questions on linear free divisors to be considered in this article are also central questions in the case of hyperplane arrangements:
\begin{enumerate}
\item\label{60} Which prehomogeneous vector spaces define linear free divisors? 
\item\label{61} Which linear free divisors satisfy the logarithmic comparison theorem?
\end{enumerate}

We investigate these questions in general and in the (combinatorial) special case of linear free divisors that arise as discriminants in quiver representation spaces, so-called quiver linear free divisors.

Concerning question~\eqref{60}, we give a representation theoretic characterization of (quiver) linear free divisors in Theorems~\ref{3} and \ref{12}.
Our main results in this direction, stated in Theorems~\ref{40} and \ref{44}, can be summarized as follows.

\begin{thm}\label{62}
(1) For a linear free divisor with reductive group, 
the number of its irreducible components coincides with the number of 
irreducible summands of the corresponding prehomogeneous space.

(2) Quiver linear free divisors occur only in the representation spaces of
quivers without (oriented or unoriented) cycles. 
\end{thm}

In Section~\ref{63} we cover special cases for which we can give a complete description of all 
\begin{enumerate}
\item irreducible linear free divisors with reductive group, or equivalently,
\item linear free divisors with reductive group and $1$-dimensional center, and all
\item linear free divisors with Abelian group.
\end{enumerate}

Our investigations concerning question~\eqref{61} are motivated by the following result and conjecture.
Denote by $j\colon U=V\setminus D\into V$ the inclusion of the complement.
One says that the \emph{logarithmic comparison theorem} holds for $D$ if the de~Rham morphism
\[
\xymat{
\Omega_X^\bullet(\log D)\ar[r]&j_*\Omega_U^\bullet\simeq\RR j_*\KK_U
}
\]
is a quasi--isomorphism.
The name is motivated by Grothendieck's comparison theorem (\cite[Thm.~2]{Gro66}) which asserts that the de~Rham morphism 
\[
\xymat{
\Omega_X^\bullet(*D)\ar[r]&j_*\Omega_U^\bullet\simeq\RR j_*\KK_U
}
\]
is a quasi--isomorphism.

Following earlier work of Castro-Jim\'enez et al.\ in \cite{CGHU07}, Narv\'aez-Macarro~\cite[Rem.~1.7.4]{Nar08} proved the following sufficient condition for the logarithmic comparison theorem to hold.

\begin{thm}\label{64}
Locally weakly quasihomogeneous free divisors satisfy the logarithmic comparison theorem.
\end{thm}

The following conjecture has been established for Koszul free divisors and for free divisors in dimension up to $3$ \cite[Cor.~1.5, Thm.~1.6]{GS06}.

\begin{cnj}\label{66}
If a free divisor satisfies the logarithmic comparison theorem then it must be strongly Euler homogeneous.
\end{cnj}

The homogeneity properties in Theorem~\ref{64} and Conjecture~\ref{66} are defined and characterized in Section~\ref{65}.
Our main results in this direction, stated in Theorems~\ref{46}, \ref{34}, and \ref{50}, can be summarized as follows.
We refer to a quiver linear free divisor as tame if it is defined by a tame quiver as defined in Section~\ref{67}.

\begin{thm}
All quiver linear free divisors are strongly Euler homogeneous.
They are locally weakly quasihomogeneous at points whose
corresponding representation is not regular.
All tame quiver linear free divisors are locally weakly quasihomogeneous.
\end{thm}

Using Theorem~\ref{64}, we derive the following

\begin{cor}
The logarithmic comparison theorem holds for tame quiver linear free divisors.
\end{cor}

In Section~\ref{68}, we study quiver linear free divisors with respect to the reflection functors.
We summarize our main results on this subject, formulated in Theorem~\ref{31} and Corollary~\ref{37}, in the following theorem, whose second statement is relevant only in case of a regular generic representation, see \S\ref{84}.

\begin{thm}
Reflection functors preserve the class of quiver linear free divisors. 
Modulo reflection functors and adding arrows with associated head and tail dimensions equal to $1$, each equivalence class contains a divisor represented by an oriented bipartite graph.
\end{thm}

In Theorem~\ref{39}, we give a simple formula for the degrees of the irreducible components of a quiver linear free divisor.

\subsection*{Acknowledgments}

We would like to thank the anonymous referee for careful reading and many valuable suggestions.
MG and MS would like to thank David Mond and the University of Warwick for their hospitality in July 2009.
MS is grateful to Anthony Kable for helpful conversations.

\section{Linear free divisors}

\subsection{Definition and characterization}\label{16}

Consider a reduced hypersurface $D\subset\AA^n_\KK=V$, and let $\cO=\cO_V$ be the sheaf of regular functions.
We use coordinates $x=x_1,\dots,x_n$ on $V$ and denote by $\del=\del_1,\dots,\del_n$, where $\del_i=\del/\del x_i$, the corresponding partial derivatives.
One calls $D$ a \emph{linear free divisor} if the sheaf $\Der(-\log D)$ of logarithmic vector fields along $D$ (cf.~\cite{Sai80}) is freely generated by the germs of globally defined linear vector fields.
By Saito's criterion \cite[Thm.~1.8.(ii)]{Sai80}, such a $D$ is the zero locus of the reduced homogeneous polynomial
\begin{equation}\label{42}
f=\det(\delta_i(x_j))\in\KK[V]
\end{equation}
of degree $n$, where $\delta_1,\dots,\delta_n$ is any basis of the Lie algebra 
\begin{equation}\label{95}
\fl:=\Gamma(V,\cL)_0=\ideal{\delta_1,\dots,\delta_n}_\KK,\quad\cL:=\Der(-\log D),
\end{equation}
of global linear logarithmic vector fields.
We call such as basis $\delta_1,\dots,\delta_n$ a \emph{Saito basis} for $D$.
Conversely, for any $n$-dimensional Lie algebra of linear vector fields $\fl$, $f$ defines a linear free divisor if it is reduced.

In more abstract terms, $D$ is a linear free divisor if the inclusion $\fl\into\Gamma(V,\cL)$ gives rise to an isomorphism of free $\cO$-modules
\[
\fl\otimes_\KK\cO\cong\cL.
\]
Then $\bigwedge^n\fl$ and $\bigwedge^n\left(\Gamma(V,\Der_\KK\cO)_{-1}\right)$ are one dimensional subspaces of $\bigwedge^n\Gamma(V,\Der_\KK\cO)$ and $f\in \KK[V]$ is defined by
\[
f\cdot\bigwedge^n\left(\Gamma(V,\Der_\KK\cO)_{-1}\right)=\bigwedge^n\fl
\]
up to a constant multiple.

In terms of algebraic group actions, linear freeness can be described as follows.
For a general $D$, define $G$ to be the largest connected subgroup of $\GL(V)$ preserving $D$.
Denote by $\rho\colon G\to\GL(V)$ the inclusion.
By \cite[Lem.~2.2]{GMNS09} the infinitesimal action of $G$ gives rise to an isomorphism
\begin{equation}\label{26}
\omega_0\colon\fg\to\fl,\quad A\mapsto x^tA^t\del,
\end{equation}
where $\fg$ is the Lie algebra of $G$, inducing a map
\begin{equation}\label{1}
\omega\colon\fg\otimes\cO\to\cL.
\end{equation}
So $D$ is linear free if and only if the latter map is an isomorphism.
Brion gave the following characterization of linear freeness in the unpublished notes \cite{Bri07b}.

\begin{thm}\label{3}
$D$ is linear free if and only if the following conditions hold:
\begin{enumerate}
\item\label{3a} $V\setminus D$ is a unique $G$-orbit, and the corresponding isotropy groups are finite.
\item\label{3b} The smooth part in $D$ of each irreducible component of $D$ is a unique $G$-orbit, and the corresponding isotropy groups are extensions of finite groups by the multiplicative group $\GG_m$.
\end{enumerate}
\end{thm}

\begin{proof}
By reflexivity of $\cL$ and since $\Sing D$ has codimension $2$, \eqref{1} being an isomorphism can be checked on $V\setminus\Sing D$.
In the terminology of \cite{Bri07a}, this condition means that the pair $(V\setminus\Sing D,D\setminus\Sing D)$ is log parallelizable under $G$.
So the ``only if'' statement follows from Proposition~2.1.2 in loc.~cit. applied to the (smooth) logarithmic strata $S_0=V\setminus D$ and the irreducible components $S_1,\dots,S_k$ of $D\setminus\Sing D$.
In our situation, this proposition states that $(V\setminus\Sing D,D\setminus\Sing D)$ is log parallelizable under $G$ along $S\in\{S_0,\dots,S_k\}$ exactly if $S$ is a unique $G$-orbit and the differential $d\rho_x$ of the normal action $\rho_x$ of $G_x$ at $x\in S$ is an isomorphism.
In case these equivalent conditions hold true, the proof yields for each $x\in S_j$, $j=1,\dots,k$, an isomorphism of exact sequences
\begin{equation}\label{4}
\xymat{
0\ar[r] & \fg_x\ar[r]\ar[d]^{d\rho_x} & \fg\ar[r]\ar[d]^{\omega_x} & \fg/\fg_x\ar[r]\ar@{^(->}[d]^{\overline\omega_x} & 0\\
0\ar[r] & N_{S_j/V,x}\ar[r]\ar@{=}[d] & \cL/\fm_x\cL\ar[r]\ar@{=}[d] & T_xS_j\ar[r]\ar@{=}[d] & 0\\
& \KK x_1\del_1 & \KK x_1\del_1\oplus\bigoplus_{i-2}^n\KK\del_i & \bigoplus_{i-2}^n\KK\del_i
}
\end{equation}
where $D$ is defined by $x_1$ in the completion $\hat\cO_x$.
The map $d\rho_x$ is the differential of the normal action and $\omega_x$ and $\overline\omega_x$ are induced by $\omega$ from \eqref{1}.
We shall get back to this diagram after the proof.

For the ``if'' statement one has to show that $d\rho_x$ is an isomorphism for $x\in V\setminus\Sing(D)$, using \cite[Prop.~2.1.2]{Bri07a} again.
This is clear for $x\in S_0$ by assumption~\eqref{3a}.
For $x\in S_j$, $j=1,\dots,k$, the normal action of the reductive subgroup $\GG_m\subseteq G_x$ from assumption~\eqref{3b} is induced by an action on a line $L$ transversal to $D$ at $x$ through some character.
Since $L\cap S_0$ is dense in $L$, this character can not be trivial in view of \eqref{3a}.
\end{proof}

Consider a linear free $D\subset V$ defined by $f\in\KK[V]$.
Then $(G,\rho,V)$ is a prehomogeneous vector space and $f$ is a relative invariant polynomial associated to some (non-trivial) character $\chi\in\X^*(G)$ by
\begin{equation}\label{5}
f(\rho(g)x)=\chi(g)f(x).
\end{equation}
This equality can be differentiated to
\begin{equation}\label{6}
\omega_0(A)(f)=\langle d\rho(A)x,\nabla f\rangle=d\chi(A)f.
\end{equation}
Define $H=\GL(V)_f^\circ\subset G$ where $\GL(V)_f=\ker\chi$ is the isotropy group of $f$ and denote by $\fh=\ker d\chi$ the Lie algebra of $H$.
The analogue of $\cL$ in \eqref{1} for $\fg$ replaced by $\fh$ is then
\[
\cH=\Der(-\log f):=\Der(-\log D)\cap\ker(df).
\]

\begin{dfn}
We call $D$ \emph{reductive} or \emph{Abelian} if $H$, or equivalently $G$, has the corresponding property. 
We call $D$ \emph{semisimple} if $H$ is semisimple.
\end{dfn}

Note that $D$ is semisimple or Abelian if and only if $\fh$, has the corresponding property.
However reductivity of $D$ can not be seen from the Lie algebra structure of $\fg$ only; it is equivalent to complete reducibility of the representation $d\rho\colon\fg\to\gl(V)$.
Note also that $D$ is semisimple if and only if $D$ is reductive and $V$ is irreducible.

\begin{prp}\label{8}
Let $D$ be a linear free divisor in $V$.
Then the $G$-orbits in Theorem~\ref{3}.\eqref{3b} are also $H$-orbits with finite isotropy groups.
For $t\ne0$, the level sets $D_t=\Z(f-t)\subset V\setminus D$ are closed $H$-orbits with finite isotropy groups.
\end{prp}

\begin{proof}
As in \cite[Lem.~2.2]{GMNS09}, one shows that $\omega_0(\fh)=\Gamma(V,\cH)_0$.
Then $\omega$ from \eqref{1} induces an isomorphism
\begin{equation}\label{7}
\omega\colon\fh\otimes\cO\to\cH
\end{equation}
and hence $\omega_x(\fh)=\cH/\fm_x\cL$.
In \eqref{4}, we may assume that $x_1$ is the image of $f$.
Then, for each $x\in S_j$, $j=1,\dots,k$, $T_x(\rho(H)x)=\omega_x(\fh)=\bigoplus_{i-2}^n\KK\del_i=T_xS_j$ and hence $\fh_x=0$.
This proves the first statement.


As $D_t\subset V\backslash D=S_0$ for $t\in\KK^*$, $D_t$ is $(n-1)$-dimensional and connected.
Thus $\rho(H)x=D_t$ and the second statement follows.
\end{proof}

\subsection{The reductive case}

\begin{prp}\label{15}
Let $D$ be a reductive linear free divisor in $V$.
Then each $H$-orbit in $D$ contains $0$ in its closure.
In particular, $\{0\}$ is the only closed $H$-orbit in $D$, or $G$-orbit in $V$.
\end{prp}

\begin{proof}
For the action of $H$ on $D$, any invariant homogeneous polynomial must be constant by Proposition~\ref{8}.
So $D$ is the null cone and $0$ is in the closure of each $H$-orbit, see \cite[\S8.1]{Dol03}.
\end{proof}

As the character $\chi$ is non-trivial it is surjective and hence $f\colon V\to\KK$ is surjective.
The following corollary is therefore a consequence of Proposition~\ref{15} and the quotient criterion in \cite[Satz~II.3.4]{Kra84}.

\begin{cor}
If $D$ is a reductive linear free divisor then $f\colon V\to\KK$ is an $H$-quotient with null cone $D$.
In particular, $\KK[V]^H=\KK[f]$.
\end{cor}

\begin{lem}\label{85}
For a reductive linear free divisor $D$ with group $G$, the number of irreducible components of $D$ equals the dimension of the center of $G$.
In particular, $D$ is irreducible if and only if $H$ is semisimple.
\end{lem}

\begin{proof}
We assume that $G$ is reductive, so $G=Z\cdot S$ is an almost direct product of the connected center $Z$ and the derived group $S$ which is semisimple. 
Each irreducible component $D_i=\ol S_i$, $i=1,\dots,k$, of $D$ is defined by a relative invariant polynomial $f_i\in\KK[V]$ with an associated character $\chi_i$ as in \eqref{5}.
These characters form a basis of the group $\X_0^*(G)$ of all characters 
corresponding to relative invariants (\cite[\S4, Lem.~4]{SK77}),
and the product of the $f_i$ is a reduced equation of $D$.
For any point $p\in S_0=V\backslash D$, this group $\X_0^*(G)$ is the subgroup of $\chi\in\X^*(G)$ with $\ker\chi\supseteq G_p$ (which is independent of the choice of $p\in S_0$ since all the corresponding isotropy groups are conjugate). 
For the orbit map $G\to Gp=S_0$ identifies the $G_p$-invariant functions on $G$ with functions on $S_0$. 
So a character $\chi$ that is trivial on $G_p$ gives rise to a function $f$ on $S_0$, which, as is easy to check, is then a relative invariant with character $\chi$.

Taking the $\#G_p$-th power, induces the second inclusion in the chain 
\[
\X^*_0(G)\subseteq\X^*(G)\into\X^*_0(G),
\]
which proves the second equality in 
\begin{equation}\label{91}
k=\rk\X_0^*(G)=\rk\X^*(G)=\rk\X^*(Z)=\dim Z.
\end{equation}
The third equality is proved similarly using 
finiteness of $Z\cap S$ and that $\X^*(S)=1$ by semisimplicity of $S$.
\end{proof}

\begin{thm}\label{40}
For a reductive linear free divisor $D\subset V$ with group $G$, the irreducible $G$-modules in $V$ are pairwise non-isomorphic, and their number equals the number of irreducible components of $D$.
\end{thm}

\begin{proof}
We continue using the notation from the proof of Lemma~\ref{85}.
Consider the decomposition of $V$ into irreducible $G$-modules
\begin{equation}\label{43}
V=V_1\oplus\cdots\oplus V_\ell,
\end{equation}
and hence, $G\subset\GL(V_1)\times\cdots\times\GL(V_\ell)$.
Denote by $Z_i\cong\GG_m$ the center of $\GL(V_i)$, and by $\fz_i$ its Lie algebra.
Then $Z$ acts by a character $\sigma_i$ on each $V_i$ and $S\subseteq\SL(V_1)\times\cdots\times\SL(V_\ell)$, since the latter is the derived group of $\GL(V_1)\times\cdots\times\GL(V_\ell)$.
We obtain a homomorphism
\begin{equation}\label{41}
Z/(Z\cap S)\cong G/S\into\frac{\GL(V_1)\times\cdots\times\GL(V_\ell)}{\SL(V_1)\times\cdots\times\SL(V_\ell)}\cong Z_1\times\cdots\times Z_\ell\cong\GG_m^\ell
\end{equation}
where $Z\cap S$ is finite.
In particular, $k\le\ell$ by \eqref{91} and \eqref{41}, which is indeed a well known fact for more general reductive prehomogeneous vector spaces.

To see the opposite inequality in our case, it suffices to show that $Z_i\cong 1\times\cdots\times1\times Z_i\times1\times\cdots\times1\subset\GL(V)$ lies in $G$ for all $i=1,\dots,\ell$, which makes \eqref{41} an isomorphism.
To this end, let $\Gamma$ denote the centralizer of $G$ in $\GL(V)$.
As $\Gamma$ permutes the $G$-orbits, it must preserve the open orbit $S_0=V\backslash D$ and hence $D$.
Thus, $\Gamma\subset G$, by definition of $G$, and hence
\[
Z_1\times\cdots\times Z_\ell\subset\Gamma=Z.
\]
In particular, $\Gamma=Z$ is commutative and $Z\cap S$ in \eqref{41} is trivial.
But then the $V_i$ must be pairwise non-isomorphic.
Indeed, $Z_i\times 1$ does not commute with switching factors in $V_i\oplus V_i$.
\end{proof}

\begin{rmk}\label{93}
Theorem~\ref{40} and its proof applies to the more general case where $(G,\rho,V)$ is a reductive prehomogeneous vector space with (possibly non-reduced) discriminant $D$ if $\rho(G)$ contains its centralizer in $\GL(V)$.
In particular this holds if $\rho(G)$ contains all linear transformations preserving $D$.
\end{rmk}

\begin{exa}
Consider $G=\GL_n(\KK)$ acting on $V=\gl_n(\KK)$ by left-multiplication.
The discriminant is the non-reduced divisor $D$ defined by $\det^n$. 
But as $G$-module, $V$ decomposes as a direct sum of the $n$ ``column subspaces''. 
So the equality of Theorem \ref{40} does not hold in general for discriminants in prehomogeneous vector spaces. 
In view of Remark~\ref{93}, what fails here is that the right-multiplications, which do preserve $D$, are not all in $G$.
However, in this example the number of components of $D$, \emph{counting multiplicity}, is still equal to the number of irreducible summands in the representation. 
\end{exa}

\subsection{Special cases}\label{63}

Based on Theorem~\ref{40} and the classification of irreducible prehomogeneous 
vector spaces by Sato and Kimura~\cite{SK77}, we can give a complete description of all semisimple (i.e.\ irreducible reductive) linear free divisors.

In \cite[\S2]{SK77}, Sato and Kimura introduce the notion of \emph{castling transformation}. 
Let $(\SL_n,\Lambda_1,V(n))$ denote the standard $n$-dimensional representation of $\SL_n$.
Two representations $(G,\rho,V)$ and $(G',\rho',V')$ are said to be castling transformations of one another if there exists a third representation $(\tilde G,\tilde\rho, V(m))$ such that 
\begin{equation}\label{75}
(G,\rho,V)\simeq(\tilde G\times\SL_n,\tilde\rho\otimes\Lambda_1,V(m)\otimes V(n))
\end{equation}
and
\begin{equation}\label{79}
(G',\rho',V')\simeq(\tilde G\times\SL_{m-n},\tilde\rho^*\otimes\Lambda_1,V(m)^*\otimes V(m-n))
\end{equation}
where $\tilde\rho^*$ is the contragredient representation of $\tilde\rho$ on the dual vector space $V(m)^*$ of $V(m)$. 
Two representations $(G,\rho,V)$ and $(G',\rho',V')$ are said to be in the same {\em castling class}, written $(G,\rho,V)\sim (G',\rho',V')$, if they are related by a finite number of castling transformations. 
In \cite[\S4, Prop.~18]{SK77}, it is shown that there is a one-to-one correspondence between relative invariants of representations related in this way, which, moreover, preserves the property of irreducibility. 
In addition, when $(G,\rho,V)$ and $(G',\rho',V')$ are related as in \eqref{75} and
\eqref{79}, and relative invariants $f$ and $f'$ correspond to one another, 
there is an integer $d$ such that $\deg f=nd$ and $\deg f'=(m-n)d$. 

We briefly review the proof from \cite{SK77}.
Identify $V(m)\otimes V(n)$ with the space $M(m,n)$ of $m\times n$ matrices, and $V(m)^*\otimes V(m-n)$ with $M(m,m-n)$. 
The representations $\tilde\rho\otimes\Lambda_1$ and $\tilde\rho^*\otimes\Lambda_1$ are defined by
\[
(g_1,g_2)\cdot (x_{i,j})=g_1(x_{i,j})g_2^t,\quad(g_1,g_2)\cdot (y_{i,j})=(g_1^t)^{-1}(y_{i,j})g_2^t,
\]
respectively.
Let $\SI(\tilde\rho\otimes\Lambda_1)$ and $\SI(\tilde\rho^*\otimes\Lambda_1)$
be the algebras of polynomial relative invariant for $\tilde\rho\otimes\Lambda_1$ and for $\tilde\rho^*\otimes\Lambda_1$ respectively. 
Any member $f$ of $\SI(\tilde\rho\otimes\Lambda_1)$ must in particular be an absolute invariant for $\Lambda_1$, and thus, by e.g.\ \cite[p.~45]{Wey97}, must be a polynomial $F$ in the $n\times n$ minors of $(x_{i,j})$. 
One should think of $F$ as a polynomial in the homogeneous coordinate ring of the Pl\"ucker embedding of $\Gr(n,m)$ in $\PP\bigwedge^n\KK^m$.
Similarly, any member of $\SI(\tilde\rho^*\otimes\Lambda_1)$ must be a polynomial in the $(m-n)\times(m-n)$ minors of $(y_{i,j})$. 
For any sequence $I=i_1,\dots,i_n$ of integers in $\{1,\dots,m\}$, let $X_{i_1,\dots,i_n}$ denote the determinant of the matrix formed by rows $i_1,\dots,i_n$ of $(x_{i,j})$. 
Similarly, let $Y_{j_1,\dots,j_{m-n}}$ denote the determinant of rows $j_1,\dots,j_{m-n}$ of $(y_{i,j})$. 
There is a natural signed bijection between these two sets of minors,
\[
X_I\mapsto Y'_I:=\sign(I,I^c)Y_{I^c}
\]
where $|I|=n$, $I^c$ is the complement of $I$ in $\{1,\dots, m\}$, and $\sign(I,I^c)$ is the sign of the permutation $I,I^c$ of $1,\dots,m$.
This bijection gives rise to an algebra isomorphism
\begin{equation}\label{80}
\Phi\colon\KK[(X_I)_{|I|=n}]\to \KK[(Y_J)_{|J|=m-n}]
\end{equation}
sending $F((X_I))$ to $F((Y'_I))$. 
It is well defined because the relations between the $X_I$ are the same as the relations between the $Y'_J$: 
both are the Pl\"ucker relations, generating the ideals of the (isomorphic) embeddings of the Grassmannians $\Gr(n,m)$ and $\Gr(m-n,m)$ in $\PP\bigwedge^n\KK^m$ and in $\PP\bigwedge^{m-n}\KK^m$ respectively. 
Sato and Kimura show that $\Phi$ is $\tilde G$-equivariant, and from this,
\cite[\S4, Prop.~18]{SK77} follows.

The following proposition summarizes results explicitly and implicitly contained in \cite{SK77}. 

\begin{prp}\label{77}
Suppose that $(G,\rho,V)\sim(G',\rho',V')$. 
Then the following statements holds true:
\begin{enumerate}
\item\label{77a} The generic isotropy subgroups of $(G,\rho,V)$ and $(G',\rho',V')$ are isomorphic.
\item\label{77b} $\dim G-\dim V=\dim G'-\dim V'$.
\item\label{77c} If $(G,\rho,V)$ is a prehomogeneous vector space then so is $(G',\rho',V')$.
\item\label{77d} If $(G,\rho,V)$ is a prehomogeneous vector space and the complement of the open orbit is a linear free divisor with group $G$, then the same goes for $(G',\rho',V')$, mutatis mutandis. 
The number of irreducible components of these divisors are the same.
\end{enumerate}
\end{prp}

\begin{proof}\
\begin{asparaenum}

\item is proved in \cite[\S2, Prop.~9]{SK77}.

\item follows from the construction, and   

\item follows from \eqref{77a} and \eqref{77b}.

\item Suppose that $\dim G=\dim V$ and $(G,\rho,V)$ is prehomogeneous. 
Denote the complements of the open orbits in $V$ and $V'$ by $D$ and $D'$. 
If $D$ is a linear free divisor with group $G$, then $\dim G=\dim V=\deg D$, and so by \eqref{77b}, $\dim G'=\dim V'$. 
The defining equation of the discriminant in a prehomogeneous vector space is the product of generators of the semigroup of polynomial relative invariants. 
The isomorphism $\Phi$ of \eqref{80} induces an isomorphism between the respective semigroups of polynomial relative invariants, multiplying each degree by $(m-n)/n$. This shows that also $\deg D'=\dim V'$.
So $D'$ is a linear free divisor with group $G'$.

\end{asparaenum}
\end{proof}

In \cite{SK77}, Sato and Kimura classified all irreducible prehomogeneous vector spaces up to castling transformations. 
We have seen in Theorem \ref{40} that semisimple (i.e.\ irreducible reductive) linear free divisors live in irreducible representations of their group. 
Thus, up to castling transformations, every semisimple linear free divisor appears as the complement of the open orbit in one of the prehomogeneous vector spaces classified in \cite{SK77}. 
We obtain 

\begin{thm}\label{92}
Up to castling transformations, there are only four semisimple (i.e.\ irreducible reductive) linear free divisors:
\begin{enumerate}
\item $D=\{0\}\subset\KK$ with $H=\{e\}$.
\item $D=\Z(y^2z^2-4xz^3-4y^3w+18xyzw-27w^2x^2)\subset\KK^4$ with $G=\GL_2(\KK)$ (and $H=\SL_2(\KK)$) acting naturally on $S^3\KK^2\cong\KK^4$ from \cite[\S5, Prop.~6]{SK77} (see also \cite[Ex.~1.4.(2), Tab.~6.1]{GMNS09}).
\item $D\subset\KK^{12}$ with $G=\SL_3(\KK)\times\GL_2(\KK)$ (and $H=\SL_3(\KK)\times\SL_2(\KK)$) from \cite[\S5, Prop.~11]{SK77}.
\item $D\subset\KK^{40}$ with $G=\SL_5(\KK)\times\GL_4(\KK)$ (and $H=\SL_5(\KK)\times\SL_4(\KK)$) from \cite[\S5, Prop.~15]{SK77}.
\end{enumerate}
\end{thm}

The next result describes the opposite extreme of Abelian linear free divisors.

\begin{thm}
The normal crossing divisors are the only Abelian linear free divisors.
\end{thm}

\begin{proof}
Let $D$ be a linear free divisor.
By \cite[Thm.~6.1]{GMNS09}, $\fl$ has a basis $\sigma_1,\dots,\sigma_s,\nu_1,\dots,\nu_{n-s}$ where the $\sigma_i$ are simultaneously diagonalizable and the $\nu_i$ are nilpotent.
It is sufficient to assume only that $[\sigma_i,\nu_j]=0$.
Then there is a coordinate system $x_1,\dots,x_n$ on $V$ such that the $\sigma_i$ are diagonal and the $\nu_i$ are block diagonal with respect to the largest blocks in which each $\sigma_i$ has a single eigenvalue.
In each such block, the $\nu_i$ can be made simultaneously strictly triangular.
Assume that there is a block of size at least $2\times 2$.
Then the defining equation $f$ of $D$ from \eqref{42} is the determinant of a matrix of the form
\[
\begin{pmatrix}
x_1 & x_2 & \cdots \\
0 & 0 & \cdots \\
\vdots & \vdots & \\ 
0 & 0 & \cdots \\
0 & x_1 & \cdots \\
\vdots & \vdots &  \\ 
0 & x_1 & \cdots \\
\end{pmatrix},
\]
and hence divisible by $x_1^2$.
As $D$ is reduced, all blocks have size $1\times 1$ and hence $s=n$.
But the  $\sigma_1,\dots,\sigma_n$ are linearly independent, and hence we can assume that $\sigma_i=x_i\partial_{x_i}$ and then $f=x_1\cdots x_n$ defines a normal crossing divisor $D$.
\end{proof}

Computer calculations of Saito bases for linear free divisors typically result in vector fields in which each coefficient is a scalar multiple of a single coordinate function. 
This is easily explained by the following lemma, which reveals also a less obvious property. 

\begin{lem}\label{87}
Let $T$ be the torus consisting of the set of diagonal matrices in $G$, and let $\lambda_i\colon T\to\KK^*$ be the character assigning to $t\in T$ its $i$'th diagonal element.  
Suppose that $\lambda_i\neq\lambda_j$ for $i\neq j$.
Then there exists a basis for $\fg$ consisting of matrices with at most one entry in each row and column. 
\end{lem} 

\begin{proof}
Under the adjoint action of $T$, $\fg$ decomposes into a direct sum of eigenspaces. 
If $a=(a_{i,j})_{1\leq i,j\leq n}\in\fg$ is an eigenvector for $\Ad$ then $(\Ad_t(a))_{i,j}=\lambda_i(t)\lambda_j(t)^{-1}a_{i,j}$, and so the quantity $\lambda_i(t)\lambda_j(t)^{-1}$ is the same for all $i,j$ such that $a_{i,j}\neq 0$. 
If $a_{i,j}\neq 0 \neq a_{i,k}$ for some $i$ and some  $j\neq k$ then $\lambda_j(t)=\lambda_k(t)$ for all $t$, contradicting the hypothesis. 
Similarly if $a_{i,j}\neq 0\neq a_{k,j}$ for some $j$ and some $i\neq k$.
Thus an eigenbasis of $\fg$ has the required property.
\end{proof}

As any torus can be diagonalized, Lemma~\ref{87} applies in particular to a maximal torus in $G$.
Both a high multihomogeneity and a high diversity of weights contribute to fulfill the hypothesis of Lemma~\ref{87}. 
For instance, a single homogeneity with all weights different, or a maximal multihomogeneity with all weights zero or one (normal crossing divisor) suffice to apply the lemma.
Lemma~\ref{87} clearly implies the following

\begin{prp}\label{88}
Under the hypothesis of Lemma~\ref{87}, $\cL$ has a Saito basis as in \eqref{95} with the following two properties:
\begin{enumerate}
\item Each coefficient of each vector field in the basis is a scalar multiple of a single coordinate function.
\item For each vector field in the basis, no two of its non-zero coefficients are divisible by the same coordinate function.
\end{enumerate}
\end{prp}

\begin{prp}
Proposition \ref{88} applies to all irreducible reductive linear free divisors.
\end{prp}

\begin{proof} 
It is straightforward to check, using the details of the representations given in \cite{SK77}, that the hypothesis of Lemma~\ref{87} holds for each of the three irreducible linear free divisors listed in Theorem~\ref{92}. 
The proposition will be proved by showing that if two linear free divisors are in the same castling class, then the hypothesis of Lemma \ref{87} holds for both, or for neither. 
To see this, notice that this hypothesis certainly holds for the standard $n$-dimensional representation $(\SL_n,\Lambda_1,V(n))$ of $\SL_n$.
Let $(G,\rho,V)$ be any representation. 
Then it is easy to see that the following three statements are equivalent:
\begin{enumerate}
\item\label{93a} The hypothesis of Lemma~\ref{87} holds for $(G,\rho,V)$.
\item\label{93b} The hypothesis of Lemma~\ref{87} holds for $(G\times\SL_n,\rho\otimes\Lambda_1,V\otimes V(n))$.
\item\label{93c} The hypothesis of Lemma~\ref{87} holds for $(G,\rho^*,V^*)$.
\end{enumerate}
For example, to see the equivalence of \eqref{93a} and \eqref{93b}, let $T_G$ and $T_{n-1}$ be maximal tori in $G$ and $\SL_n$ respectively, and pick bases $u_1,\dots, u_m$ and $v_1,\dots,v_m$ of $V$ and $V(n)$ respectively, that diagonalize these tori.
Then $u_i\otimes v_j$, $1\leq i\leq m$, $1\leq j\leq n$, is a basis for $V\otimes V(n)$ and $T_G\times T_{n-1}$ is a maximal torus in $G\times\SL_n$. 
Denote by $\lambda_i$ and $\mu_j$ the diagonal characters on $G$ associated with the basis elements $u_i\in V$ and $v_j\in V(n)$, and by $\lambda_{i,j}$ the diagonal  character associated to the basis element $u_i\otimes v_j$ of $V\otimes V(n)$.
Then for $(s,t)\in T_G\times T_{n-1}$ we have $\lambda_{i,j}(s,t)=\lambda_i(s)\mu_j(t)$. From this, and the fact that the hypothesis holds for $(\SL_n,\Lambda_1,V(n))$, the equivalence of \eqref{93a} and \eqref{93b} is clear.

From the equivalence of the three statements it now follows that the hypothesis of Lemma~\ref{87} holds for $\bigr(G\times\SL_n,\rho\times\Lambda_1,V\otimes V(n)\bigl)$ if and only if it holds for $\bigl(G\times\SL_{m-n},\rho^*\otimes\Lambda_1,V^*\otimes V(m-n)\bigr)$. 
In view of the definition of castling given in \eqref{75} and \eqref{79}, this immediately establishes that it holds for all the members of a castling class, or for none. 
\end{proof}

In Proposition~\ref{89} below we use Theorem~\ref{40} to show that Proposition~\ref{88} applies in the case of linear free divisors arising from quiver representations.

\subsection{Euler homogeneity and weak quasi-homogeneity}\label{65}

A hypersurface $D$ in $V$ is \emph{Euler homogeneous} at $p\in D$ if $f\in\fm_p\langle\frac{\del f}{\del x_i}\rangle$ for some, and hence any, local defining equation $f$ of $D$ at $p$. 
This condition can be rephrased by saying that $\chi(f)=f$ for some vector field $\chi\in\Der(-\log D)_p$ with $\chi(p)=0$.
Such a vector field is called an Euler vector field at $p$.
If the preceding condition holds for all $p\in D$, then $D$ is called 
\emph{strongly} Euler homogeneous.

\begin{prp}\label{18}
A linear free divisor $D$ is Euler homogeneous at $p\in D$ if and only if $G_p^\circ\not\subseteq H$.
\end{prp}

\begin{proof}
For a linear free divisor $D$, Euler homogeneity at $p$ is equivalent to $\omega_0(\id)(p)\in\omega_0(\fh)(p)$ \cite[Lem.~7.5]{GMNS09}.
This means that $\omega_0(\id)(p)=\omega_0(A)(p)$ for some $A\in\fh$ and hence $\omega_0(\id-A)(p)=0$ or $\id-A\in\fg_p$.
As $\fg=\KK\id\times\fh$, this is equivalent to $\fg_p\not\subseteq\fh$.
The claim follows.
\end{proof}

\begin{rmk}\label{35}
The ``if'' part of Proposition~\ref{18} holds more generally for a prehomogeneous vector space $(G,\rho,V)$, a relative invariant $f\in\KK[V]$ with associated character $\chi$, $H=\ker\chi$, and $D=\Z(f)$.
\end{rmk}

An algebraic hypersurface $D\subset V$ is \emph{weakly quasihomogeneous} at $p\in D$, if $f=(x_1,\dots,x_n)A(x)(\frac{\del f}{\del x_i})^t\in\fm\langle\frac{\del f}{\del x_i}\rangle$ with $A(0)$ having only non negative eigenvalues for some, and hence any, local defining equation $f$ of $D$ at $p$.
We notice this condition implies Euler homogeneity and that, by \cite[Lem.~2.4]{GS06}, $A(0)$ has at least one positive eigenvalue.

By \cite{Sch07}, weak quasi-homogeneity means that, after a suitable analytic coordinate change, $D$ is defined at $p$ by a quasihomogeneous polynomial with respect to non-negative weights.
If the preceding condition holds for all $p\in D$, then $D$ is called locally weakly quasihomogeneous.
If the attribute ``weakly'' is dropped, this means that all weights have to be strictly positive.

We call a representation $\rho\colon G\to\GL(V)$ \emph{positive} 
if the convex hull of its weights does not contain $0$, and \emph{non-negative} if the interior of the convex hull of its weights does not contain $0$.

\begin{prp}\label{47}
Let $D$ be a linear free divisor in $V$.
Then $D$ is quasihomogeneous (weakly quasihomogeneous) at $p$ if the normal representation $\rho_p\colon G_p\to\GL(T_pV/T_p(Gp))$ is positive (non-negative).
\end{prp}

\begin{proof}
Assume that $\rho_p$ is positive and set $N_p=T_pV/T_p(Gp)$.
Then there is a $1$-PSG $\lambda\colon\GG_m\to G_p\to\GL(N_p)$ that is positive on the weights of $\rho_p$.
As $\GG_m\subset G_p$ is reductive and preserves both $T_pV$ and $T_p(Gp)$, there is an affine linear transversal $S_p\subset V$ to the orbit $Gp$ at $p$ such that $\lambda$ restricts to $\lambda'\colon\GG_m\to\GL(S_p)$.
Then $S_p\to N_p$ is an isomorphism and hence $\lambda'$ is positive on its weights.
This means that $\lambda'$ defines an Euler vector field at $p$ with strictly positive weights for $D'=D\cap S$.
But locally analytically $(D,p)=\left(\AA_\KK^{\dim(Gp)},0\right)\times(D',p)$ and hence with $D'$ also $D$ is quasihomogeneous at $p$.
This statement about weak quasihomogeneity is proved analogously.
\end{proof}

\section{Quiver representations}

\subsection{The category of quiver representations}

A \emph{quiver} $(Q_0,Q_1,s,t)$ consists of a finite set $Q_0$ of vertices, a finite set $Q_1$ of arrows, and source and target maps $s,t\colon Q_1\to Q_0$.
Without loss of generality, we shall always assume that $Q$ is a connected graph.
A $Q$-representation $M=(V_i,f_\alpha)$ consists of a family of finite $\KK$-vector spaces $V_i$, $i\in Q_0$, and a family of homomorphisms $f_\alpha\in\Hom(V_{s\alpha},V_{t\alpha})$, $\alpha\in Q_1$.
We write $\ddim(M)=(\dim V_i)_{i\in Q_0}\in\NN^{Q_0}$ for its dimension vector.
Homomorphisms $\phi=(\phi_i)_{i\in Q_0}\colon M\to N$, $\phi_i\in\Hom(V_i,W_i)$, of two $Q$-representations $M=(V_i,f_\alpha)$ and $N=(W_i,g_\alpha)$ are defined by the obvious commutativity condition $g_\alpha\circ\phi_{s\alpha}=\phi_{t\alpha}\circ f_\alpha$.

To a $Q$-representation $M=(V_i,f_\alpha)$ on can associate the vector space $V=\bigoplus_{i\in Q_0}V_i$, maps $f_\alpha\colon V\to V$, $\alpha\in Q_1$, induced by $f_\alpha$ and the projection maps $f_i\colon V\to V_i\into V$, $i\in Q_0$.
These maps satisfy the relations $f_i^2=f_i$ and $f_\alpha f_{s\alpha}=f_\alpha=f_{t\alpha}f_\alpha$ and all other products are zero.
The path algebra $\KK Q$ of $Q$ is the associative algebra on the symbols $e_i$, $i\in Q_0$, and $e_\alpha$, $\alpha\in Q_1$, subject to the preceding relations.
It is finite $\KK$-dimensional if and only if $Q$ has no oriented loops.
Note that
\begin{equation}\label{9}
1=\sum_{i\in Q_0}e_i
\end{equation}
is a decomposition into orthogonal idempotents.
Clearly $V$ becomes a $\KK Q$-module by $e_i\mapsto f_i$ and $e_\alpha\mapsto f_\alpha$.

This construction gives rise to the following equivalence \cite[Prop.~1.2.2]{Bri08}.

\begin{prp}
The category of $Q$-representations is equivalent to that of left $\KK Q$-modules.
\end{prp}

This is the starting point for the study of quiver representations based on general module theory over an associative finite $\KK$-dimensional algebra. 
For background in this area we suggest \cite{Bri08} and \cite{CB92}, and the more encyclopedic \cite{ARS95}. 
Since it more directly concerns linear free divisors, we also cite \cite{BM06}, though the proof of the main theorem therein (\cite[Cor.~5.5]{BM06}) is considerably simplified below (see Theorem \ref{111}). 

It follows from the Fitting Decomposition that $M$ is an indecomposable $\KK Q$-module if and only if $\End_{\KK Q}M=\KK\id_M\oplus I$ where $I$ is a nilpotent ideal \cite[Lem.~1.3.3]{Bri08}.
From this one deduces the following result \cite[Thm.~1.3.4]{Bri08}.

\begin{thm}\label{13}
Let $M$ be a finite $\KK$-dimensional $\KK Q$-module and $M=\bigoplus_{i=1}^sr_iM_i$ its decomposition into indecomposables $M_i$ (with $M_i\not\cong M_j$ if $i\neq j$).
Then $\End_{\KK Q} M=B\oplus I$ where $B\cong\prod_{i=1}^s\Mat_{r_i\times r_i}(\KK)$ and $I$ is a nilpotent ideal.
\end{thm}

The automorphism group $\Aut_Q(M)$ of a $Q$-representation $M$ is affine open in its Lie algebra $\End_Q(M)$ and hence connected.
This yields the following description of $\Aut_Q(M)$ \cite[Prop.~2.2.1]{Bri08}.

\begin{prp}\label{10}
Let $M$ be a $Q$-representation with decomposition into indecomposables $M=\bigoplus_{i=1}^sr_iM_i$.
Then $\Aut_Q(M)\cong U\rtimes\prod_{i=1}^s\GL_{r_i}(\KK)$ where $U$ is a closed normal unipotent subgroup. 
\end{prp}

The category of $\KK Q$-modules is \emph{hereditary}, that is, every submodule of a projective $\KK Q$-module is itself projective. 
For this reason the projective dimension of every $\KK Q$-module is at most $1$.
In fact, by \eqref{9}, the left $\KK Q$-modules $P(i)=\KK Qe_i$ are projective indecomposables and each left $\KK Q$-module $M$ has a standard projective resolution
\[
\xymat{
0\ar[r] & \bigoplus_{\alpha\in Q_1}P(t\alpha)\otimes e_{s\alpha}M\ar[r]^u & \bigoplus_{i\in Q_0}P(i)\otimes e_iM\ar[r]^-v & M\ar[r] & 0
}
\]
where $u(a\otimes m)=a\alpha\otimes m-a\otimes\alpha m$ on the $\alpha$-summand and $v$ is the canonical multiplication map \cite[Prop.~1.4.1]{Bri08}. 
Since all higher $\Ext$ groups vanish, instead of $\Ext^1_Q(M,N)$ we write $\Ext_Q(M, N)$. 

Using the isomorphism $\Hom_{\KK Q}(P(e),M)\cong eM$ defined by evaluation at the idempotent $e$, one obtains the Ringel exact sequence \cite[Cor.~1.4.2]{Bri08} 
\begin{equation}\label{11}
\xymat@C=15pt@R=10pt{
0\ar[r] & \Hom_Q(M,N)\ar[r] & \bigoplus_{i\in Q_0}\Hom(V_i,W_i)\ar`r[d]`[d]`[ldd]`[dd][dd]^-{c_{M,N}} & & \\
&&&&\\
& & \bigoplus_{\alpha\in Q_1}\Hom(V_{s\alpha},W_{t\alpha})\ar[r] & \Ext_Q(M,N)\ar[r] & 0 
}
\end{equation}
associated to two $Q$-representations $M=(V_i,f_\alpha)$ and $N=(W_i,g_\alpha)$, where
\[
c_{M,N}((\phi_i)_{i\in Q_0})=(\phi_{t\alpha}f_\alpha-g_\alpha\phi_{s\alpha})_{\alpha\in Q_1}.
\]
For $M=N$ it induces an exact sequence 
\begin{equation}\label{14}
\xymat@C=15pt@R=10pt{
0\ar[r] & \End_Q(M)\ar[r] & \bigoplus_{i\in Q_0}\End(V_i)\ar`r[d]`[d]`[ldd]`[dd][dd]^-{c_{M}} & & \\
&&&&\\
& & \bigoplus_{\alpha\in Q_1}\Hom(V_{s\alpha},V_{t\alpha})\ar[r] & \Ext_Q(M)\ar[r] & 0
}
\end{equation}
where $c_M=c_{M,M}$ and $\Ext_Q(M)=\Ext_Q(M,M)$ denotes the self-extensions of $M$. 
In \cite{BM06} this sequence is interpreted in terms of deformation theory, and $\Ext_Q(M)$ is identified as the $T^1$ (set of isomorphism classes of infinitesimal deformations) of the $\KK Q$-module $M$. 
See also the discussion in Section~\ref{110}. 

The \emph{Euler form} on $\QQ^{Q_0}$ is defined by
\[
\langle m,n\rangle_Q=\sum_{i\in Q_0}m_in_i-\sum_{\alpha\in Q_1}m_{s\alpha}n_{t\alpha}=\langle m,En\rangle
\]
where $E$ is the Euler matrix defined by
\[
E=I-A,\quad A=(\#\{\alpha\in Q_1\mid s\alpha=i,\ t\alpha=j\})_{i,j}.
\]
The \emph{Tits form} is the associated quadratic form 
\begin{equation}\label{69}
q_Q(n)=\langle n,n\rangle_Q
\end{equation}
and the \emph{Cartan matrix} $C=E+E^t$ defines the associated symmetric bilinear form
\[
(m,n)_Q=\langle m,n\rangle_Q+\langle n,m\rangle_Q=\langle m,Cn\rangle.
\]
For dimension vectors $m=\ddim(M)$ and $n=\ddim(N)$ it follows from \eqref{11} and \eqref{14} that 
\begin{align}
\langle m,n\rangle_Q&=\dim\Hom_Q(M,N)-\dim\Ext_Q(M,N),\nonumber\\
q_Q(m)&=\dim\End_Q(M)-\dim\Ext_Q(M)\label{76}.
\end{align}
By abuse of notation, we shall write $\langle M,N\rangle_Q$ for $\langle\ddim(M),\ddim(N)\rangle_Q$ and similarly for $q_Q$.
\begin{arxiv}
For $M=(V_i,f_\alpha)$, we shall sometimes abbreviate $M(i)=V_i$, $M(\alpha)=f_{\alpha}$, $sM=\bigoplus_{\alpha\in Q_1}V_{s\alpha}$, and $tM=\bigoplus_{\alpha\in Q_1}V_{t\alpha}$.
\end{arxiv}

\subsection{Quiver representation spaces}\label{110}

Choosing bases of the $V_i$, each $Q$-representation $V$ with fixed dimension vector $d=(\dim V_i)_{i\in Q_0}$ can be considered as a point $(x_\alpha)_{\alpha\in Q_1}$ in the representation space
\begin{equation}\label{45}
\Rep(Q,d)=\bigoplus_{i\to j\in Q_1}\Hom(\KK^{d_i},\KK^{d_j})=\bigoplus_{i\to j\in Q_1}\Mat_{d_j\times d_i}(\KK)
\end{equation}
on which the group $\GL(Q,d)=\prod_{i\in Q_0}\GL_{d_i}(\KK)$ acts by
\[
(g_i)_{i\in Q_0}(x_\alpha)_{\alpha\in Q_1}=(g_{t\alpha}x_\alpha g_{s\alpha}^{-1})_{\alpha\in Q_1}.
\]
In this way the $\GL(Q,d)$-orbits in $\Rep(Q,d)$ correspond to the isomorphism classes of $Q$-representations with dimension vector $d$.
Note that the central subgroup $\KK^*\id\subset\GL(Q,d)$ acts trivially on $\Rep(Q,d)$, so $\GL(Q,d)$ acts via $\PGL(Q,d)=\GL(Q,d)/\KK^*\id$, and that, by the connectedness of $Q$, $\PGL(Q,d)\into\GL(\Rep(Q,d))$.
The isotropy group of $M\in\Rep(Q,d)$ for the actions of $\GL(Q,d)$ and $\PGL(Q,d)$ can be identified with $\Aut_Q(M)$ and $\Aut_Q(M)/(\KK^*\id_M)$.

In the the sequence \eqref{14}, $c_M$ can be identified with the differential at $\id$ of the orbit map $\GL(Q,d)\ni g\mapsto g\cdot M \in\Rep(Q,d)$.
This identifies $\End_Q(M)$ with the isotropy Lie algebra of $M$ in $\gl(Q,d)$ and $\Ext_Q(M)$ with the normal space in $\Rep(Q,d)$ to the orbit $\GL(Q,d)M$ at $M$, that is,
\begin{equation}\label{74}
\Ext_Q(M)=\Rep(Q,d)/T_M(\GL(Q,d)M)
\end{equation}
Note that $c$ corresponds to $\omega_0$ in \eqref{26}.

The following consequences of Theorem~\ref{3} and Proposition~\ref{10} are inspired by \cite{Bri07b}.
Let $d$ be a sincere dimension vector on a quiver $Q$, that is, $d_i\ge1$ for all $i\in Q_0$.
We shall say that $(Q,d)$ \emph{defines a linear free divisor} $D$, and call $D$ a \emph{quiver linear free divisor}, if the discriminant of the action of $\GL(Q,d)$ on $\Rep(Q,d)$ is a linear free divisor $D$ with associated group
\[
G:=\PGL(Q,d)
\]
as in \S\ref{16}.
Any such $D$ is a reductive linear free divisor.

A $Q$-representation $M$ is called a \emph{brick} if $\End_Q(M)=\KK\id_M$ and $d$ is called a \emph{Schur root} if $\Rep(Q,d)$ contains a brick.

\begin{thm}\label{12}
$(Q,d)$ defines a linear free divisor if and only if the following conditions are fulfilled:
\begin{enumerate}
\item\label{12a} $q_Q(d)=1$
\item\label{12b} $d$ is a Schur root.
\item\label{12c} $\dim_\KK\End_Q(N)=2$ for any minimal degeneration $N\in\Rep(Q,d)$.
\end{enumerate}
Condition \eqref{12c} can be replaced by one of the following:
\begin{enumerate}
\item[(\ref{12c}')] $\dim_\KK\Ext_Q(N)=1$ for any minimal degeneration $N\in\Rep(Q,d)$. 
\item[(\ref{12c}'')] All minimal degenerations lie in orbits of codimension $1$. 
\end{enumerate}
\end{thm}

\begin{proof}
Recall that $\Aut_Q(N)$ is connected and that $\Aut_Q(N)/(\KK^*\id)$ can be identified with the isotropy group $G_N$ of $N\in\Rep(Q,d)$.
By the following Remark~\ref{17}, we can assume $N$ in \eqref{12c} to be decomposable.
Then, by Proposition~\ref{10}, \eqref{12c} is equivalent to $\Aut_Q(N)\cong\GG_m^2$.

Assume that $(Q,d)$ defines a linear free divisor. 
Then, for $M\in V\setminus D$, the second condition in Theorem~\ref{3}.\eqref{3a} becomes $\Aut_Q(M)=\KK^*\id_M$, and hence $M$ is a brick and \eqref{12b} holds true.
The first condition in Theorem~\ref{3}.\eqref{3a} implies 
\begin{align*}
q_Q(d)&=\dim\GL(Q,d)-\dim\Rep(Q,d)\\
&=\dim\GL(Q,d)-\dim(\GL(Q,d)M)=\dim\Aut_Q(M)
\end{align*}
and hence \eqref{12a}.
Similarly, the conditions in Theorem~\ref{3}.\eqref{3b} imply \eqref{12c} and \eqref{12a}.

By reversing these equalities, \eqref{12a} and \eqref{12b} show that a brick $M$ has an orbit of dimension $\dim(\GL(Q,d)M)=\dim\Rep(Q,d)$, which is then open.
Note that, in an irreducible variety, open and dense are equivalent properties of orbits and that such orbits are unique.
Then \eqref{12a} and \eqref{12c} force the complement $D$ of the open orbit to be a hypersurface and imply the conditions in Theorem~\ref{3}.\eqref{3b}.

The last statement follows from \eqref{14}.
\end{proof}

\begin{rmk}\label{17}
By Kac's theorem \cite[Thm.~4.3]{KR86}, conditions \eqref{12a} and \eqref{12b} in Theorem~\ref{12} can be rephrased by saying that $d$ is a real Schur root and in this case there is a unique orbit of indecomposables.
By \eqref{14} and Theorem~\ref{13}, this orbit is open and consists of bricks.
\end{rmk}

\begin{cor}\label{23}
Let $(Q,d)$ define a linear free divisor $D$.
Then the following statements hold:
\begin{enumerate}
\item\label{23a} Up to isomorphism, there is a unique indecomposable $Q$-representation $M$ with $\dim M=d$.
Moreover, $M$ is a brick in the open orbit $\Rep(Q,d)\setminus D$.
\item\label{23b} Any minimal degeneration $N$ is the direct sum $N_1\oplus N_2$ of two non-isomorphic bricks. 
\end{enumerate}
\end{cor}

\begin{proof}\
\begin{enumerate}
\item This follows from Remark~\ref{17}.
\item This follows from Theorem~\ref{13} and Theorem~\ref{12}.\eqref{12c}.
\end{enumerate}
\end{proof}

A representation $N$ of a quiver $Q$ is \emph{rigid} if it has no non-trivial deformations; this is equivalent to the infinitesimal condition $\Ext_Q(N)=0$. 
The following result refines the statement in Corollary~\ref{23}.\eqref{23b}.

\begin{prp}\label{73}
Let $(Q,d)$ define a linear free divisor $D$, and let $N\in\Rep(Q,d)$ be a minimal degeneration.
Then $N=N_1\oplus N_2$ for two uniquely determined rigid submodules $N_1$ and $N_2$ whose dimension vectors $d_1$ and $d_2$ are real Schur roots.
After perhaps permuting $N_1$ and $N_2$, they satisfy the following conditions:
\begin{gather}
\label{90}\End_Q(N_i)=\KK\id_{N_i},\ \Hom_Q(N_1,N_2)=0=\Hom_Q(N_2,N_1),\\
\label{78}\Ext_Q(N_i)=0,\ \Ext_Q(N_1,N_2)=0,\ \Ext_Q(N_2,N_1)=\KK.
\end{gather}
\end{prp} 

\begin{proof}
By Corollary~\ref{23}.\eqref{23b} and Theorem~\ref{12}.\eqref{12c}, we have $N=N_1\oplus N_2$ satisfying the conditions in \eqref{90}.
The representations $N_1$ and $N_2$ are unique because $\Hom_Q(N_i,N)=\Hom_Q(N_i,N_1)\oplus\Hom_Q(N_i,N_2)$ is $1$-dimensional for $i\in\{1,2\}$, and therefore must be generated by the canonical inclusion $N_i\into N_1\oplus N_2=N$. 

The representation $N_i$ is rigid for $i\in\{1,2\}$ because for any representation $N'_i\in\Rep(Q,d_i)$ in a sufficiently small neighborhood of $N_i$, the representation $N'_i\bigoplus N_j$, where $\{i,j\}=\{1,2\}$, is in $D$ by Corollary~\ref{23}.\eqref{23b} and therefore in the orbit of $N$, since this orbit is open in $D$. 
The resulting isomorphism $N'_i\bigoplus N_j\to N$ sends $N'_i$ to $N_i$ by the unicity of $N_1$ and $N_2$.

By \eqref{74}, $\Ext_Q(N)=\bigoplus_{i,j=1}^2\Ext_Q(N_i,N_j)$ is $1$-dimensional, and $\Ext_Q(N_i,N_i)=0$ for $i\in\{1,2\}$ by the rigidity of $N_i$.
Therefore the conditions in \eqref{78} are satisfied after reordering $N_1$ and $N_2$.
Now \eqref{76} shows that $q_Q(N_i)=1$ for $i\in\{1,2\}$ and by rigidity this implies that $d_1$ and $d_2$ are real Schur roots.
\end{proof}

The representation $N$ in Proposition~\ref{73} is, of course, a split extension of $N_2$ by $N_1$.
Where can we find a non-split extension?
The answer is that every representation not in $D$ is the total space of such an extension. 
In fact, using \eqref{14} we can explicitly construct the line $L$ referred to in the proof of Theorem~\ref{3}, that is, a $G_N$-invariant line $L$ meeting $D$ transversely at $N$. 
Here, each point of $L\setminus D$ will be a non-split extension of $N_2$ by $N_1$.

Choose $\theta\in\Hom_{\KK}(sN_2,tN_1)$ so that its image under the inclusion $\Hom_{\KK}(sN_2,tN_1)\to\Hom_{\KK}(tN,hN)$, generates $\coker c_{N,N}$.
Construct a new representation $N(\lambda\theta)$ of $Q$, depending on the complex parameter $\lambda$, as an extension of $N_2$ by $N_1$:

\begin{equation}\label{81}
\xymat@C=.6in@R=.4in{
0\ar[r]&N_1(s\alpha)\ar[r]\ar[d]^{N_1(\alpha)}&N_1(s\alpha)\oplus N_2(s\alpha)\ar[r]\ar[d]^{\begin{pmatrix}N_1(\alpha)&\lambda\theta\\0&N_2(\alpha)\end{pmatrix}}&N_2(s\alpha)\ar[r]\ar[d]^{N_2(\alpha)}&0\\
0\ar[r]&N_1(t\alpha)\ar[r]&N_1(t\alpha)\oplus N_2(t\alpha)\ar[r]&N_2(t\alpha)\ar[r]&0
}
\end{equation}

The representation in the center is $N(\lambda\theta)$.
Evidently $N(\lambda\theta)=N$ when $\lambda=0$.
Since the tangent space to the line $L:=\{N(\lambda\theta):\lambda\in\KK\}$ is spanned by $\theta$, which does not belong to $\img(c_{N,N})=T_ND$, $L$ is a complement to $T_ND$. 
It follows that except for a finite number of values of $\lambda$, $N(\lambda\theta)\notin D$ and the extension \eqref{81} is not split, and therefore generates $\Ext^1_Q(N_2,N_1)$.

\begin{lem}\label{96}
The line $L=\{N(\lambda\theta):\lambda\in\KK\}$ is invariant under the action of the isotropy group $G_N\subset\GL_{Q,d}$ of $N$.
\end{lem}

\begin{proof} 
As observed above, $G_N=G_{N_1\oplus N_2}$ is isomorphic to $\KK^*\times\KK^*$ acting by scalar multiplication on each of the two summands.
Let $(u,v)\in G_N$.
The diagram
\[
\xymat@C=.6in@R=.4in{
N_1(s\alpha)\oplus N_2(s\alpha)
\ar[r]^{\begin{pmatrix}u&0\\0&v\end{pmatrix}}
\ar[d]_{\begin{pmatrix}N_1(\alpha)&\lambda\theta\\0&N_2(\alpha)\end{pmatrix}}  
&N_1(s\alpha)\oplus N_2(s\alpha)\ar[d]^{\begin{pmatrix}N_1(\alpha)
&uv^{-1}\lambda\theta\\
0&N_2(\alpha)\end{pmatrix}}\\
N_1(t\alpha)\oplus N_2(t\alpha)\ar[r]^{\begin{pmatrix}u&0\\0&v\end{pmatrix}}
&N_1(t\alpha)\oplus N_2(t\alpha)
}
\]
commutes; it follows that $(u,v)\cdot N(\lambda\theta)=N(uv^{-1}\lambda\theta)$.
Thus $G_N$ acts on $L$.
\end{proof}
This construction gives a self-contained proof of another sufficient condition for the discriminant in $\Rep(Q,d)$ to be a linear free divisor:

\begin{prp}\label{82}
Let $d$ be a real Schur root of the quiver $Q$, and let $D\subset\Rep(Q,d)$ be the complement of the open orbit. 
If there is a $\GL(Q,d)$-orbit which is open in the irreducible component $D_i$ of $D$, then the equation $f$ of the discriminant $D$ (see \eqref{42}) is reduced along $D_i$. 
If this holds for each component of $D$ then $D$ is a linear free divisor.
\end{prp}

\begin{proof} 
We continue to use the notation and assumptions of Lemma~\ref{96}.  
Regard $\lambda$ as a coordinate on $L$. 
Identify $G_N$ with $\KK^*\times\KK^*$ as described above, and consider the curve $\sigma(u)=(u,1)\in G_N$.
As $(u,1)\cdot N(\lambda\theta)=N(u\lambda\theta)$, we have
\begin{equation}\label{83}
\frac{d}{du}\left(\sigma(u)\cdot N(\lambda\theta)\right)\vert_{u=1}
=\begin{pmatrix}0&\lambda\theta\\0&0\end{pmatrix}
=\lambda\del_{\lambda}N(\lambda\theta).
\end{equation}
Now consider $\sigma'(1)$ as an element of $\gl(Q,d)$ via the inclusion $G_N\subset \GL(Q,d)$. 
Via the infinitesimal action, it gives rise to a linear vector field $\delta$ on $\Rep(Q,d)$; 
\eqref{83} means that $\delta$ restricts to $\lambda\del_{\lambda}$ on $L$.

Since the corank at $N $ of the map $c_{N,N}$ is $1$, its image generates the tangent space $T_ND$. 
So we can choose vectors $v_i\in\gl(Q,d)$, $i=1,\dots,r-1$, whose images under $c_{N,N}$ form a basis for $T_ND$, where $r=\dim\Rep(Q,d)$. 
These extend to vector fields $\delta_i=\omega_0(v_i)$, $i=1,\dots,r-1$ (see \eqref{26}).
The determinant $\Delta:=\det(\delta,\delta_1,\dots,\delta_{r-1})$ is reduced at $N$ since it is a unit times $\lambda$ when restricted to $L$. 
Since it is not identically zero, the vectors $\sigma'(1),v_1,\dots,v_{r-1}$ form a basis for $\gl(Q,d)$, so $\Delta$ is a non-zero complex multiple of $f$. 
It follows that $f$ is reduced along $D_i$.
\end{proof}

The center of $\GL(Q,d)$ is the quotient $\prod_{i\in Q_0}Z(\GL_{d_i}(\KK))/\KK^*\id$, and thus has dimension equal to $\#Q_0-1$. 
Therefore it is an immediate consequence of Lemma~\ref{85} that for a quiver linear free divisor, the number of irreducible components is one less than the number of nodes in the quiver. 
This statement, which is valid in greater generality, is attributed by Schofield in \cite{Sch91} to V.~Kac. 
Together with Theorem~\ref{40}, it gives the following strong restriction on those quivers $Q$ which can define a linear free divisor.

\begin{thm}\label{44}
If $(Q,d)$ defines a linear free divisor then $Q$ is a tree.
\end{thm}

\begin{proof}
Each summand of $\Rep(Q,d)$ in \eqref{45} is an irreducible $\GL(Q,d)$-module.
So their number equals $\ell=\# Q_1$.
The number $k$ of irreducible components of $D$ equals $\# Q_0-1$.
By Theorem~\ref{40}, we have $\# Q_1=\ell=k=\# Q_0-1$ and hence $Q$ is a tree, as it is connected.
\end{proof}

\begin{cor}\label{36}
Let $(Q,d)$ define a linear free divisor.
Then there is a partition $Q_0=\bigcup_{k=0}^mQ_0^k$ such that $Q_1\subset\bigcup_{k=1}^mQ_0^{k-1}\times Q_0^{k}$.\qed
\end{cor}

\begin{prp}\label{89}
Proposition \ref{88} applies to quiver linear free divisors.  
\end{prp}

\begin{proof}
We have $\Rep(Q,d)=\bigoplus_{\alpha\in Q_1}\Hom(\KK^{d_{s\alpha}},\KK^{d_{t\alpha}})$.
Take the entries $c^{(\alpha)}_{i,j}$ in the generic matrix, as coordinates on $\Hom(\KK^{d_{s\alpha}},\KK^{d_{t\alpha}})$. 
The $c^{(\alpha)}_{i,j}$, as $\alpha$ ranges over $Q_1$, give a set of coordinates on $\Rep(Q,d)$. 
For $k\in Q_0$, let $T^{(k)}$ be the set of diagonal matrices in $\GL_{d_k}(\KK)$, and let $\lambda^{(k)}_j\colon T^{(k)}\to\KK^*$ be the $j$'th diagonal element. 
The group $\GL(Q,d)$ contains $T=\prod_{k\in Q_0}T^{(k)}/\KK^*\id$ as a maximal torus; we denote its members by $\# Q_0$-tuples $\tau=(\tau^{(k)})_{k\in Q_0}$. 
Let $\lambda^{(\alpha)}_{i,j}$ be the character on $T$ corresponding to the coordinate $c^{(\alpha)}_{i,j}$ on $\Rep(Q,d)$. 
That is, if $\{e^{(\alpha)}_{i,j}\}$ is the basis for $\Rep(Q,d)$ dual to the coordinate system $\{c^{(\alpha)}_{i,j}\}$, 
\[
\lambda^{(\alpha)}_{i,j}(\tau)=c^{(\alpha)}_{i,j}(\tau\cdot e^{(\alpha)}_{i,j}).
\]
Then we have 
\[
\lambda^{(\alpha)}_{i,j}(\tau)=\lambda^{(t\alpha)}_i(\tau^{(t\alpha)})(\lambda^{(s\alpha)}_j(\tau^{(s\alpha)}))^{-1}
\]
For fixed $\alpha$, $\lambda^{(\alpha)}_{i_1,j_1}\neq\lambda^{(\alpha)}_{i_2,j_2}$ if $i_1\neq i_2$ (or $j_1\neq j_2$), since then $\lambda^{(s\alpha)}_{i_1}\neq\lambda^{(s\alpha)}_{i_2}$ (or $\lambda^{(t\alpha)}_{j_1}\neq\lambda^{(t\alpha)}_{j_2}$), and the 
$\lambda^{(s\alpha)}_i$ and $\lambda^{(t\alpha)}_j$ are algebraically independent if $s\alpha\ne t\alpha$ and $\# Q_1>1$.
The former holds by Theorem~\ref{44}, and the case $\# Q_1=1$ is trivial.
For $\alpha\neq\beta$, $\lambda^{(\alpha)}_{i_1,j_1}\neq\lambda^{(\beta)}_{i_2,j_2}$ for all $i_1,j_1,i_2,j_2$, since by Theorem~\ref{44} either $s\alpha\neq s\beta$ or $t\alpha\neq t\beta$.
\end{proof}

Notice that in the proof we make use of the fact that a quiver giving rise to a linear free divisor has no pairs of arrows with the same source and the same target, and no arrow whose source is the same as its target, both of which follow from Theorem \ref{44}.

\subsection{Dynkin quivers}

A quiver $Q$ is a \emph{Dynkin quiver} if its underlying graph is a Dynkin diagram of type $A_n$, $D_n$, $E_6$, $E_7$ or $E_8$. 
Gabriel proved in \cite{Gab72}, \cite{Gab80} that the Dynkin quivers are precisely those of ``finite representation type'': that is, such that for any $d$, $\Rep(Q,d)$ contains only finitely many orbits (see also \cite{BGP73} for a simpler account). 
If $Q$ is a Dynkin quiver and $d$ is any dimension vector then $\Rep(Q,d)$ therefore contains an open orbit. 
If moreover  $q_Q(d)=1$, then $\dim\PGL(Q,d)=\dim \Rep(Q,d)$; it follows that the Lie algebra of linear vector fields coming from the infinitesimal action of $\GL(Q,d)$ on $\Rep(Q,d)$ has dimension equal to that of $\Rep(Q,d)$. 
Thus all that is needed for the discriminant $D$ to be a linear free divisor is that the determinant of the Saito matrix should be reduced. 
Again, because $Q$ is a Dynkin quiver, in each irreducible component of $D$ there is an open orbit.
It follows by Lemma~\ref{82} that the discriminant is reduced.  
This gives a concise proof of the following theorem from \cite[Cor.~5.5]{BM06} (in which the first statement is well-known).

\begin{thm}\label{111}
If $Q$ is a Dynkin quiver and $q_Q(d)=1$ then $d$ is a real Schur root and the discriminant in $\Rep(Q,d)$ is a linear free divisor. 
\end{thm}

\subsection{Relative invariants}

For a quiver $Q$ and a real Schur root $d\in\NN^{Q_0}$, $\Rep(Q,d)$ 
is prehomogeneous for the action of $\GL(Q,d)$ by Remark~\ref{17}.
Denote by $M(Q,d)^\perp$ the right orthogonal category of the generic representation $M(Q,d)$ in $\Rep(Q,d)$, that is, the full subcategory of all $M$ for which
\[
\Hom(M(Q,d),M)=0=\Ext_Q(M(Q,d),M).
\]
For $M\in M(Q,d)^\perp$, $f_{d,M}$ defined by $f_{d,M}(x)=\det(c_{x,M})$ using \eqref{11}, where $x$ are coordinates in $\Rep(Q,d)$, is a relative invariant as considered by Schofield in \cite[\S1]{Sch91}.
We denote the character associated to $f_{d,M}$ by $\chi_{d,M}$.
Similarly one defines relative invariants $f_{M,d}$ with associated characters $\chi_{M,d}$ for $M$ in the left orthogonal category $^\perp M(Q,d)$ of $M(Q,d)$ defined by
\[
\Hom(M,M(Q,d))=0=\Ext_Q(M,M(Q,d)).
\]
Assuming that $Q$ has no oriented loops and that $d$ is sincere, Schofield~\cite[Thm.~4.3]{Sch91} showed that if $M_1,\dots,M_m$, $m=\# Q_1-1$, are the simple objects in $^\perp M(Q,d)$, then $f_{M_1,d},\dots,f_{M_m,d}$ form a basis of the ring of all polynomial relative invariants; the analogous statement holds true for $M(Q,d)^\perp$.
In particular, $f_{M_1,d}\cdots f_{M_r,d}$ is a reduced equation for the discriminant $D$ in $\Rep(Q,d)$.
On the other hand, an equation for $D$ is given by
\begin{equation}\label{27}
f(x)=f_d(x)=\det(\bar c_x)
\end{equation}
with character $\chi=\chi_d$, where $\bar c_x$ is the map induced by $c_x$ in \eqref{14} by factoring out the diagonal $\KK\id$ from $\End\left(\bigoplus_{i\in Q_0}V_i\right)$.
Thus, Saito's criterion proves 

\begin{prp}\label{97}
Let $Q$ be a quiver without oriented loops and with a sincere real Schur root $d\in\NN^{Q_0}$.
Then the discriminant $D$ in $\Rep(Q,d)$ is a linear free divisor if and only if
\[
\sum_{i=1}^r\deg(f_{M_i,d})=\deg(f_d)=\dim\Rep(Q,d).
\]
\end{prp}

The characters of $\GL(Q,d)$ are products of determinants on the factors $\GL_{d_i}(\KK)$, $i\in Q_0$, and hence $\X^*(\GL(Q,d))$ can be identified with $\ZZ^{Q_0}$.
Schofield~\cite[Lem.~1.4]{Sch91} and Buchweitz and Mond~\cite[Lem.~6.6]{BM06} proved the following

\begin{lem}\label{20}
Let $Q$ be a quiver and $d\in\NN^{Q_0}$.
We have $\chi_{d,M}=Em$ for $M\in M(Q,d)^\perp$ and $\chi_{M,d}=-E^tm$ for $M\in{^\perp M(Q,d)}$, where $m=\ddim(M)$.
Moreover, $\chi_d=(E-E^t)d$ in case $d$ is a real Schur root.
\end{lem}

We can combine this with Corollary~\ref{36} to obtain a formula for the degree of a polynomial relative invariant.

\begin{prp}\label{38}
In the situation of Corollary~\ref{36}, 
\[
\deg(f_{d,M})=\sum_{k=1}^mk\sum_{j\in Q_0^k}d_j\left(m_j-\sum_{j\to i\in Q_1\cap(Q_0^k\times Q_0^{k+1})}m_i\right)
\]
for $M\in M(Q,d)^\perp$ with $m=\ddim(M)$, where we set $Q_0^{l+1}=\emptyset$.
There is a similar formula for $\deg(f_{M,d})$ if $M\in{^\perp M(Q,d)}$.
\end{prp}

\begin{proof}
As $f_{d,M}(\lambda x)=\lambda^{\deg(f_{d,M})}f(x)$, we have $\chi_{d,M}([\lambda])=\lambda^{\deg(f_{d,M})}$ where $[\lambda]\in\PGL(Q,d)=G$ is defined by $\lambda^k\id\in\GL_{d_i}(\KK)$ for $i\in Q_0^k$, using the notation from Corollary~\ref{36}.
By Lemma~\ref{20} and Corollary~\ref{36},
\[
\chi_{d,M}=Em=m-Am=\left(m_j-\sum_{j\to i\in Q_1\cap(Q_0^k\times Q_0^{k+1})}m_i\right)_{k=1,\dots,l,\ j\in Q_0^k}.
\]
As $\det(\lambda^k\id)=\lambda^{kd_j}$ for $\lambda^k\id\in\GL_{d_j}(\KK)$, each vertex $j\in Q_0^k$ contributes a summand $kd_j\left(m_j-\sum_{j\to i\in Q_0^k\times Q_0^{k+1}}m_i\right)$ to $\deg(f_{d,M})$.
The claim follows.
\end{proof}

\subsection{Euler homogeneity}

Combined with Propositions~\ref{18} and \ref{10}, Lemma~\ref{20} gives the following characterization of Euler homogeneity for quiver linear free divisors.

\begin{lem}\label{29}
Let $(Q,d)$ define a linear free divisor $D$ and fix a $Q$-representation $M\in D$.
Then $D$ is Euler homogeneous at $M$ if and only if $M$ decomposes as $M=M_1\oplus M_2$ such that $\langle M_1,M_2\rangle_Q\ne\langle M_2,M_1\rangle_Q$.
In this case, $M_1$ and $M_2$ can be chosen to have no isomorphic indecomposable summands.
\end{lem}

\begin{proof}
By Proposition~\ref{10}, the isotropy group of $M$ with respect to $\GL(Q,d)$ is $\Aut_Q(M)\cong U\rtimes\prod_{i=1}^s\GL_{r_i}(\KK)$ where the $r_i$ are the multiplicities of the indecomposable summands of $M$ and $U$ is a closed normal unipotent subgroup.
Then also the image of $U$ under each homomorphism $\GL(Q,d)\to\GL_{d_i}(\KK)$, $i\in Q_0$, is unipotent and hence lies in the kernel of $\det$ on $\GL_{d_i}(\KK)$.
Thus $U$ lies in the kernel of any character of $\GL(Q,d)$.
Whether a character is trivial on the factor $\prod_{i=1}^s\GL_{r_i}(\KK)$ of $\Aut_Q(M)$ can be checked on the torus of diagonal matrices $T=\GG_m^s$.
Indeed, $\PGL_{r_i}(\KK)$ is connected semisimple and has no non-trivial characters. 
So restriction defines an injection
\begin{equation}\label{99}
\X^*(\Aut_Q(M))\into\X^*(T),\quad\chi\mapsto\chi_{\vert T}.
\end{equation}
For $i=1,\dots,s$, let $N_i=r_iM_i$ be the product of isomorphic indecomposable summands of $M$, as in Theorem \ref{13}, and let $m_i=\ddim(N_i)\in\NN^{Q_0}$ be the corresponding dimension vector.
An examination of the proof of Proposition~\ref{10} in \cite[Prop.~2.2.1]{Bri08} shows that the $i$th factor of $T$ corresponds to the subgroup $\KK^*\id_{N_i}\subset\GL(Q,d)$.
So the restriction of a character $\chi\in\ZZ^{Q_0}=\X^*(\GL(Q,d))$ to $T$ equals
\begin{equation}\label{100}
\chi_{\vert T}=\left(\langle m_i,e\rangle\right)_{i=1,\dots,s}\in\ZZ^s=\X^*(T).
\end{equation}
Now by Proposition~\ref{18}, Euler homogeneity of $D$ at $M$ is equivalent to $\chi_d$ being non-trivial on $\Aut_Q(M)$, or equivalently by \eqref{99} on $T$.
By \eqref{100}, this means that $\langle m_i,\chi_d\rangle\ne0$ for some $i=1,\dots,s$.
For a fixed $i\in\{1,\dots,s\}$, set $m=m_i$ and $n=d-m_i$. 
Then, using Lemma~\ref{20},
\begin{align*}
\langle m,\chi_d\rangle&=\langle m,(E-E^t)d\rangle
=\langle m,(E-E^t)(m+n)\rangle\\
&=\langle m,Em\rangle-\langle m,E^tm\rangle+\langle m,En\rangle-\langle m,E^tn\rangle
=\langle m,n\rangle_Q-\langle n,m\rangle_Q.
\end{align*}
So finally Euler homogeneity of $D$ at $M$ is equivalent to 
\[
\langle m,n\rangle_Q\ne\langle n,m\rangle_Q
\]
for some splitting $d=m+n$ with respect to non-isomorphic indecomposable summands of $M$.
\end{proof}

\begin{rmk}
The ``if'' part of Lemma~\ref{29} also holds in the situation of Remark~\ref{35}.
However, it is not clear that with $D$ also its reduction $D_\red$ is Euler homogeneous at $M$.
While the module $\Der(-\log D)$ does not change if $D$ is replaced by $D_\red$, this is not clear for the property of being an Euler vector field.
\end{rmk}

Surprisingly, a different approach proves the following

\begin{thm}\label{46}
All quiver linear free divisors are strongly Euler homogeneous.
\end{thm}

\begin{proof}
Let $(Q,d)$ define a linear free divisor $D$ and fix $M=(V,f)\in D$.
By Corollary~\ref{23}.\eqref{23a}, $M$ decomposes as $M=M_1\oplus M_2=(V_1\oplus V_2,f_1\oplus f_2)$.
Then $\Ext_Q(M)=\bigoplus_{i,j=1}^2\Ext_Q(M_i,M_j)$ is by \eqref{74} the normal space at $M$ to the $\GL(Q,d)$-orbit of $M$ in $\Rep(Q,d)$.
In the sequence \eqref{14}, it lifts to a subspace $W=\bigoplus_{i,j=1}^2W_{i,j}\subset\bigoplus_{i,j=1}^2\bigoplus_{\alpha\in Q_1}\Hom_Q(V_{i,\alpha},V_{j,\alpha}) $, such that $\{M_{\theta}\}_{\theta\in W}$, $M_{\theta}:=(V,f+\theta)$, is an affine space in $\Rep (Q,d)$ normal to this orbit.
The representation $M_{\theta}$ splits if $\theta _{1,2}=\theta _{2,1}=0$. 
As the generic representation is indecomposable, the generic $M_{\theta}$ can not split, and  we must have $\Ext_Q(M_2,M_1)\ne0$ after reordering the $M_1,M_2$.
Consider the $1$-PSG $\lambda\colon\GG_m\to\GL(Q,d)_M$ that acts by the characters $(1,0)\in\ZZ^2$ on the summands.
As it is non-zero on the weights of the normal representation, it induces an Euler vector field at $M$.
At a smooth point of $D$, this is just the vector field $\delta$ of the proof of Proposition~\ref{82}.
\end{proof}

Combining Lemma~\ref{29} and Theorem~\ref{46} we find

\begin{thm}
Let $(Q,d)$ define a linear free divisor $D$.
Then any $M\in D$ decomposes as $M=M_1\oplus M_2$ such that $\langle M_1,M_2\rangle_Q\ne\langle M_2,M_1\rangle_Q$.
In this case, $M_1$ and $M_2$ can be chosen to have no isomorphic indecomposable summands.
\end{thm}

\subsection{Local quasihomogeneity}

Under additional hypotheses, the construction used in the proof of Theorem~\ref{46} can be used to show even weak local quasihomogeneity.

\begin{prp}\label{48}
Let $(Q,d)$ define a linear free divisor.
Then $D$ is weakly quasihomogeneous at $M=(V,f)$ if $M=M_1\oplus M_2$ with $M_1\neq0\neq M_2$ and $\Ext_Q(M_2,M_1)=0$; it is quasihomogeneous if $M=M_1\oplus\cdots\oplus M_k$ with $\Ext_Q(M_i,M_j)=0$ for all $i\le j$.
\end{prp}

\begin{proof}
In the first case, the $1$-PSG from the proof of Theorem~\ref{46} is non-negative on the weights of the normal representation and the claim follows from Proposition~\ref{47}.
For the second statement, consider the $1$-PSG $\lambda\colon\GG_m\to\GL(Q,d)_M$ that acts by the characters $(k-1,k-2,\dots,1,0)\in\ZZ^k$ on the summands of a decomposition $V=V_1\oplus\cdots\oplus V_k$ corresponding to $M=M_1\oplus\cdots\oplus M_k$.
By hypothesis, it is positive on the weights of the normal representation and the claim follows again from Proposition~\ref{47}.
\end{proof}

The \emph{Auslander--Reiten translate} of a $Q$-representation $M$ is defined by 
\[
\tau M=D\Ext(M,\KK Q),
\]
where $DX=\Hom_\KK(X,\KK)$. 
It satisfies $\Ext_Q(M,N)=\Ext_Q(\tau M,\tau N)$ if $N$ is a non-projective indecomposable $Q$-representation (see \cite[\S6, Lems.~1, 2]{CB92}). 
A representation is \emph{preprojective} if $\tau^kM$ is a non-zero projective $\KK Q$-module for some $k$ (in which case $\tau^{k+1}M=0$). 
In the case of Dynkin quivers, any indecomposable $Q$-representation $M$ is preprojective.  
Set $\nu(M)=k$ for the minimal such $k$ and order the indecomposable summands $M_1,\dots,M_k$ of $M$ so that $\nu(M_1)\ge\cdots\ge\nu(M_k)$.
Consequently,
\begin{equation}\label{112}
\Ext_Q(M_i,M_j)=\Ext_Q\left(\tau^{\nu(M_i)}M_i,\tau^{\nu(M_i)}M_j\right)=0 \text{ whenever } i\ge j.
\end{equation}
Using Proposition~\ref{48}, this shows the following

\begin{thm}
Any Dynkin quiver linear free divisor is locally quasihomogeneous.
\end{thm}

\begin{rmk}
Feh\'er and Patakfalvi~\cite[\S5.2]{FP09} used the same argument to show more generally that all points in the representation space of a Dynkin quiver have positive normal representations. 
\end{rmk}

We can generalize the preceding arguments as follows.
Like $\tau$, its inverse functor $\tau'$ defined by $\tau'M=\Ext_Q(DM,\KK Q)$,  leaves $\Ext_Q$ invariant, and by close analogy with $\tau$, defines the class of \emph{preinjective} $Q$-representations.
A $Q$-representation is called \emph{regular} if it has no preprojective or preinjective indecomposable summands.

Assume that $M$ is not regular. 
If it has a preprojective summand $M_1$, we may assume $\nu(M_1)$ is minimal, and take a complementary summand $M'$. 
Then using \eqref{112} we get $\Ext_Q(M_1,M')=0$. 
Similarly if $M$ has a preinjective summand $M_2$, we may assume $\Ext_Q(M',M_2)=0$, with $M'$ a complement of $M_2$.
Then Proposition~\ref{48} implies the following

\begin{thm}\label{34}
Let $(Q,d)$ define a linear free divisor $D$.
Then $D$ is weakly quasihomogeneous at any non-regular point.
\end{thm}

\subsection{The tame case}\label{67}

Recall that a quiver $Q$ is called \emph{tame} if the Tits form $q_Q(d)$, defined in \eqref{69}, is positive and non definite.
Its kernel is then the one dimensional module $\ZZ\cdot\delta$ generated by a positive sincere dimension vector $\delta\in\NN^{Q_0}$, which is characterized by the following equivalent properties:
\[
\delta'\in\ZZ\cdot\delta
\Longleftrightarrow q_Q(\delta')=0
\Longleftrightarrow (\delta',d)_Q=\ideal{\delta',d}_Q+\ideal{d,\delta'}_Q=0.
\]

The goal of this section is to prove the following result. 

\begin{thm}\label{50}
Suppose that $(Q,d)$ defines a linear free divisor $D$. 
If $Q$ is tame then $D$ is locally weakly quasihomogeneous.
\end{thm}

For the rest of the section, we assume the hypotheses of Theorem~\ref{50}.
Let $M$ be a point in $D$ and let $M=M_1\oplus\cdots\oplus M_r$, $r\geq 2$, its decomposition into indecomposable $Q$-representations. 
According to Theorem~\ref{34}, the only case for which local weak quasihomogeneity needs to be proved is when all the $M_i$ are regular indecomposable. 
Let us recall some facts about the regular $Q$-representations, which can be found for example in \cite{CB92}:

\begin{asparaitem}

\item The Auslander--Reiten functors $\tau$ and $\tau'$ applied to a regular indecomposable $Q$-representation $X$ satisfy $X\cong\tau\tau'X\cong\tau'\tau X$. 
We define $\tau^iX$ for any  $i\in \ZZ$ by 
\[
\tau^iX=
\begin{cases}
\tau^iX\text{ if }i\geq 0,\\
\tau'^{-i}X\text{ if }i\le 0.
\end{cases}
\]

\item Define the \emph{defect} of an indecomposable $Q$-representation $X$ to be $\defect(X)=\ideal{\delta,\ddim(X)}_Q=-\ideal{\delta,\ddim(X)}_Q$. 
Then, according to \cite[\S 7, Lem.~2]{CB92}, the regular indecomposable $Q$-representations are characterized by
\[
X \text{ is regular }\Longleftrightarrow\defect(X)=0
\]

\item The regular $Q$-representations form a full Abelian subcategory.
We call its simple objects regular-simple.

\item An indecomposable regular $Q$-representation $X$ contains only a finite set of regular $Q$-subrepresentations $X_i$ which are organized in a composition series
\begin{equation}\label{70}
\{0\}=X_0\subset X_1\subset \cdots \subset X_r=X,\quad X_{r-i}/X_{r-i-1}=\tau^iT,
\end{equation}
where $T$ is a uniquely determined regular-simple $Q$-representation.
We call the $\tau^iT$ the regular-simple factors of $X$ and $T$ the top regular-simple factor.

\item Conversely for any regular-simple $Q$-representation $T\ne0$ and any $r\ge1$ there is, up to isomorphism, exactly one indecomposable regular $Q$-representation $X$ with a composition series \eqref{70} having $T$ as its top regular-simple factor.
We denote this module by $X(T,r)$ and set $X(T,0)=0$. 
Then we have exact sequences 
\[
\xymat{0\ar[r]&X(\tau^sT,r-s)\ar[r]&X(T,r)\ar[r]&X(T,s)\ar[r]&0},\quad 1\le s<r.
\]  

\end{asparaitem}

\begin{lem}\label{51}
Let $X(T_1,r_1)$, $X(T_2,r_2)$ be two regular indecomposable $Q$-representations. 
\begin{enumerate}
\item\label{51a} The space of homomorphisms $\Hom_Q(X(T_1,r_1),X(T_2,r_2))$ is non zero if and only if there are integers $a,b\in\NN$ such that $0\le a\le r_1-1$, $0\le b\le r_2-1$, $a+r_2=b+r_1$ and 
\[
\tau ^{r_1-1}T_1=\tau ^a(\tau ^{r_2-1}T_2),\quad T_1=\tau^bT_2.
\] 
\item\label{51b} The space of extensions $\Ext_Q(X(T_1,r_1), X(T_2,r_2))$ is non zero if and only if there are integers $a,b\in\NN$ such that $1\le a\le r_1$, $1\le b\le r_2$, $a+r_2=b+r_1$ and 
\[
\tau ^{r_2-1}T_2=\tau ^b(\tau ^{r_1-1}T_1),\quad T_2=\tau^aT_1.
\]
\end{enumerate}
\end{lem}

\begin{proof}\
\begin{asparaenum}

\item Consider a non zero homomorphism $\theta\colon X(T_1,r_1)\to X(T_2,r_2)$.
Its kernel $\ker\theta$ being again a regular $Q$-representation must equal $X(\tau^sT_1,r_1-s)$, for some $s\in\NN$ such that $1\le s\le r_1$.
Since $X(T_1,r_1)/X(\tau^sT_1,r_1-s)\cong X(T_1,s)$, $\theta$ induces an injective homomorphism
\[
\ol{\theta}\colon X(T_1,s)\into X(T_2,r_2)
\]
Its image, being a regular $Q$-representation, must equal $X(\tau^{b}T_2,r_2-b)$, for some integer $b\in\NN$ such that $0\le b \le r_2-1$. 
By uniqueness of the regular-simple factors, we conclude that $\tau^{s-1}T_1=\tau^{r_2-1}T_2$. 
Setting $a=r_1-s$, this proves the last assertion.

\item The second part follows directly from the first using the Auslander--Reiten formulas (\cite[\S6, Lem.~1]{CB92})
\begin{align*}
\Ext_Q(X(T_1,r_1),X(T_2,r_2))
&=D(\Hom_Q(X(T_2,r_2),\tau X(T_1,r_1)))\\
&=D(\Hom_Q(X(T_2,r_2),X(\tau T_1,r_1)))
\end{align*}

\end{asparaenum}
\end{proof}

The following periodicity result on the regular indecomposable $Q$-representations, which can be found in \cite[Lec.~9]{CB92}, is a key ingredient of the proof of Theorem~\ref{50}.

\begin{prp}\label{53}
Let $T$ be a regular-simple $Q$-representation.
Then there is a minimal integer $p=p(T)\in\NN$, the period of $T$, such that $\tau^pT\cong T$. 
Then
\begin{equation}\label{52}
\dim T+\dim\tau T+\cdots+\dim\tau^{p-1}T=\delta
\end{equation}   
and $X=X(T,r)$ is a brick if and only if $dim X\le\delta$, or equivalently $r\le p$.
\end{prp}

By this result, the set of regular-simple $Q$-representations, is the set of nodes of a disjoint union of cyclic graphs $Q_T=(T,\tau T,\dots,\tau^{p-1}T)$ in which  $p=p(T)$, and the arrows are all the $(\tau^iT,\tau^{i+1}T)$, $0\le i\le p-1$. 
Then by \eqref{70} the regular indecomposable $Q$-representations $X$ are in one-to-one correspondence with the paths $\Gamma(X)=(\tau^aT,\dots,\tau^{a+r-1}T)$ indexed by $(\tau^aT,r)$, with $\tau^aT$ in the support of some $Q_T$ and $r\in \NN$. 
The entries of $\Gamma(X)$ are the quotients $X_{r-i}/X_{r-i-1}$ in the sequence \eqref{70} with $T$ replaced by $\tau^aT$, $0\le a<p$. 
By Proposition~\ref{53}, the path $\Gamma(X)$ has no repeated node if and only if $X$ is a brick.

Furthermore, $\Ext_Q(X_1,X_2)\ne 0$ if and only if $X_1$ and $X_2$ are associated with the same $Q_T$, and by Lemma~\ref{51} we may write $\Gamma(X_i)=(\tau^{a_i}T,\dots,\tau^{a_i+r_i-1}T)$, with $a_1<a_2\le a_1+r_1\le a_2+r_2-1$: 
$\Gamma(X_1)\cup\Gamma(X_2)$ is again the support of a path, and both initial and terminal node of $\Gamma(X_2)$ are the images of their counterparts in $\Gamma (X_1)$ by translations $\tau^k\ne\id$. 
With this picture in mind, we are ready for the proof of Theorem~\ref{50}.
We shall write $d<\delta$ if $d\le \delta$ (component-wise) and $d\neq\delta$.

\begin{proof}[Proof of Theorem~\ref{50}]
Fix a regular point $M$ in $D$ with (regular) indecomposable summands $X_1,\dots,X_r$.
Then, by linearity, $0=\ideal{\delta,\dim M}_Q=\ideal{\delta,d}_Q$, which implies that the generic $Q$-representation $M(Q,d)$ is a regular indecomposable $X(T,r)$. By Remark~\ref{17} and Theorem~\ref{12}, $M(Q,d)$ is a brick with $q_Q(d)=1\ne0=q_Q(\delta)$.
Proposition~\ref{53} therefore shows that $d<\delta$ or equivalently $r<p$. 
This implies also that each $X_i=X(T_i,r_i)$ has the same property so that $r_i<p_i:=p(T_i)$.

By Proposition~\ref{48}, the claim follows if we establish the existence of a splitting $M=M_1\oplus M_2$ such that $\Ext_Q(M_2,M_1)=0$. 
Let us consider a directed graph with vertices $X_i$ and an arrow $(X_i,X_j)$ if and only if $\Ext(X_i,X_j)\ne 0$. 
Then the splitting can be obtained by grouping the $X_i$ into two non trivial summands if and only if this graph does not contain any directed loop. 
The theorem is therefore a consequence of the following Lemma~\ref{71}, and of the obvious fact that $\sum_{i=1}^k\dim X_i=\dim M<\delta$.
\end{proof}

\begin{lem}\label{71}
Let $X_1,\dots,X_k$ be regular indecomposables with $\sum_{i=1}^k\dim X_i<\delta$ and $\Ext(X_i,X_{i+1})\neq 0$ for $i=1,\dots,k-1$.
Then we have $\Ext(X_k,X_1)=0$.
\end{lem} 

\begin{proof}
According to Lemma~\ref{51}.\eqref{51b}, the sets $\Gamma_i=\Gamma(X_i)$ of all $Q$-representations $X_i$ lie in a single list $Q_T$.
Moreover, the set 
\[
\Gamma(X_1)\cup\cdots\cup \Gamma(X_k)
\]
is, by the preliminary description, a list with consecutive entries $\tau^{a_1} T,\dots,\tau ^{a_k+r_k-1}T$ starting with initial entry of $\Gamma(X_1)$ and ending with the terminal entry of $\Gamma(X_k)$. 
This list has length $a_k+r_k-a_1<p$ by \eqref{52} and the dimension hypothesis.
The result follows by applying Lemma~\ref{51}.\eqref{51b} again.
\end{proof}

\subsection{Reflection functors}\label{68}

Reflection functors, introduced by Bernstein, Gel'fand and Ponomarev in \cite{BGP73}, and independently as castling transforms by Sato and Kimura (cf.\ Section~\ref{63}), play an important role in the study of quiver representations.
We shall show explicitly that they induce functors on the class of (strongly Euler homogeneous) quiver linear free divisors.

\begin{thm}\label{31}
Reflection functors preserve the class of quiver linear free divisors.
\end{thm}

Let $Q$ be a quiver and denote by $e_k\in\ZZ^{Q_0}$ the unit vector corresponding to $k\in Q_0$.
To a vertex $k\in Q_0$ with no loop attached one associates a reflection $r_k\colon\ZZ^{Q_0}\to\ZZ^{Q_0}$ by setting $r_k(m)=m-(m,e_k)_Qe_k$.
These reflections generate the Weyl group $W_Q\subseteq\GL(\ZZ^{Q_0})$ which is of finite index in the group of transformations which leave $q_Q$ invariant.

Reflection functors are a realization on quiver representations of the above reflections on dimension vectors.
Quivers $Q$ for which $k$ is a source are transformed to quivers $Q^*=r_k(Q)$ for which $k$ is a sink, and vice versa, by reversing all arrows involving $k$. 
Dimension vectors $d\in\NN^{Q_0}$ with
\begin{equation}\label{28}
d_k\le\sum_{k\to i\in Q_1}d_i
\end{equation}
correspond to dimension vectors $d^*\in\NN^{Q^*_0}$ with $d^*_k\le\sum_{k\gets i\in Q^*_1}d^*_i$ via $d^*=r_k(d)$, or more explicitly by $d^*_k=\sum_{k\to i\in Q_1}d_i-d_k$ and $d^*_i=d_i$ for $i\ne k$.
Note that condition~\eqref{12a} in Theorem~\ref{12} is invariant under $d\leftrightarrow d^*$.
The correspondence can be extended to a correspondence between the orbits in open subsets $\Rep'(Q,d)\subset\Rep(Q,d)$ and $\Rep'(Q^*,d^*)\subset\Rep(Q^*,d^*)$ as follows.
For $M=(V_i,f_\alpha)$ and $M^*=(V^*_i,f^*_\alpha)=r_k(M)$, these are defined by the conditions that, for all $k$,
\begin{gather*}
(f_{k\to i})_{k\to i\in Q_1}\colon V_k\rightarrow\bigoplus_{k\to i\in Q_1}V_i\text{ is injective and}\\
\sum_{k\gets i\in Q^*_1}f^*_{k\gets i}\colon V^*_k\leftarrow\bigoplus_{k\gets i\in Q^*_1}V^*_i\text{ is surjective}.
\end{gather*}
Then the correspondence is given by
\begin{equation}\label{102}
V^*_k=\coker\left((f_{k\to i})_{k\to i\in Q_1}\right),\quad V_k=\ker\left(\sum_{k\gets i\in Q^*_1}f^*_{k\gets i}\right),\quad V_i=V^*_i,\ i\ne k,
\end{equation}
$f^*_{k\to i\in Q_1}$ and $f^*_{k\gets i\in Q^*_1}$ being the canonical maps, and $f^*_\alpha=f_\alpha$ for $\alpha\in Q_1\cap Q^*_1$.

We shall use the following trivial

\begin{lem}\label{21}
Let $Q$ be a quiver, $k$ a source or sink of $Q$, and $d\in\NN^{Q_0}$.
If $M\in\Rep(Q,d)$ is indecomposable then $M=e_kM$ or $M\in\Rep'(Q,d)$.
In particular, if $\#Q_0>1$ and $d$ is a sincere root then $M\in\Rep'(Q,d)$ and \eqref{28} holds true.
\end{lem}

Note that if $(Q,d)$ defines a linear free divisor then the hypotheses for the last statement in Lemma~\ref{21} are fulfilled.
Indeed, $Q$ must have a source $k$ by Theorem~\ref{44}, $d$ is sincere by definition and a Schur root by Theorem~\ref{12}.\eqref{12b}, and $\#Q_0=1$ would imply $Q_1=\emptyset$ and hence $\Rep(Q,d)=0$ by Theorem~\ref{44}.

The following result can be found in \cite[Prop.~4.1]{KR86}.

\begin{thm}\label{25}
Let $k$ be a source of a quiver $Q$.
The correspondence \eqref{102} induces a homeomorphism 
\[
\Rep'(Q,d)/\GL(Q,d)\to\Rep'(Q^*,d^*)/\GL(Q^*,d^*)
\]
preserving the isomorphism class of endomorphism rings.
\end{thm}

Let us assume that $(Q,d)$ defines a linear free divisor $D$ and fix a source $k$ of $Q$.
The case of a sink $k$ can be treated similarly.
We want to show that also $(Q^*,d^*)$ defines a linear free divisor $D^*$.
Let us denote by $Z=\Rep(Q,d)\setminus\Rep'(Q,d)$ and $Z^*=\Rep(Q^*,d^*)\setminus\Rep'(Q^*,d^*)$ the complements of the domain and range of the reflection functor.
Note that $Z$ and $Z^*$ are defined by the maximal minors of $(x_{k\to i})_{k\to i\in Q_1}$ and $\sum_{k\gets i\in Q^*_1}x^*_{k\gets i}$ and invariant under $\GL(Q,d)$ and $\GL(Q,d^*)$ respectively.
So $Z\subseteq D$ as a union of orbits excluding the open one.

Let us first assume that $d_k=\sum_{k\to i\in Q_1}d_i$.
Then $Z$ is defined by the determinant $\det((x_{k\to i})_{k\to i\in Q_1})$ of a generic matrix, and is hence an irreducible component of $D$. 
A general point in $Z$ is a direct sum of two indecomposables $N_1$ and $N_2$ subject to the conditions in Proposition~\ref{73}.
Then Lemma~\ref{21}, implies that, after perhaps permuting $N_1$ and $N_2$, $N_1=e_kN_1$ and hence $\dim N_1=1$ by indecomposability of $N_1$.
As $\Hom(N_2,N_1)=0$, this implies that $e_kN_2=0$ and hence $1=d_k=\sum_{k\to i\in Q_1}d_i$.
In other words, there a unique arrow $\alpha\colon k\to l\in Q_1$ with $s\alpha=k$ and $Z$ is defined by $x_\alpha=0$. 
Note that $Q^*$ is the full subquiver of $Q$ with $Q^*_0=Q_0\setminus\{k\}$ and $d\in\NN^{Q_0}$ restricts to $d^*\in\ZZ^{Q^*_0}$.
Then the inclusion $\iota\colon\GL(Q,d^*)\into\GL(Q,d)$ defined by $\iota(g)_k=g_l$ and $\iota(g)_i=g_i$ for $i\in Q_0^*$ identifies $\GL(Q,d^*)$ with the subgroup of $\GL(Q,d)$ that stabilizes the coordinate $x_\alpha$.
Thus $D$ is of the form $D=D^*\times\AA^1\cup\AA^{\dim D}\times\{0\}$ where $D^*$ is the linear free divisor defined by $(Q^*,d^*)$.

We summarize these arguments as follows.

\begin{prp}\label{24}
Let $(Q,d)$ define a linear free divisor $D$. 
Then the following statements hold:
\begin{enumerate}
\item\label{24a} $Q$ has both sources and sinks.
\item\label{24b} For a source (or sink) $k\in Q_0$, $d_k\le\sum_{k\to i\in Q_1}d_i$ (or $d_k\le\sum_{k\gets i\in Q_1}d_i$).
\item\label{24c} In case of equality in \eqref{24b}, there is a unique arrow $\alpha\colon k\to l\in Q_1$ with $s\alpha=k$ (or $\alpha\colon k\gets l\in Q_1$ with $t\alpha=k$) and $d_k=1=d_l$.
Note that $(Q^*,d^*)$ is then obtained from $(Q,d)$ by $Q^*_0=Q\setminus\{k\}$ and $Q^*_1=Q_1\setminus\{\alpha\}$ and $d^*=d\vert_{Q^*_0}$.
\item\label{24d} If $(Q,d)$ and $(Q^*,d^*)$ are related as in \eqref{24c}, then $(Q,d)$ defines a linear free divisor if and only if $(Q^*,d^*)$ does.
Such linear free divisors are related by $D=D^*\times\AA^1\cup\AA^{\dim D}\times\{0\}$.
\end{enumerate}
\end{prp}

\begin{cor}
Let $(Q,d)$ define a linear free divisor $D$ and assume that $d_i=d_j$ for all $i,j\in Q_0$.
Then $d_i=1$ for all $i\in Q_0$ and $D$ is a normal crossing divisor.
\end{cor}

Now let us consider the case $d_k<\sum_{k\to i\in Q_1}d_i$.
Then $Z\subset\Rep(Q,d)$ and $Z^*\subset\Rep(Q^*,d^*)$ are invariant determinantal subvarieties of codimension at least $2$.
By Theorem~\ref{25} and Remark~\ref{17}, condition~$\eqref{12b}$ in Theorem~\ref{12} is preserved under $(Q,d)\leftrightarrow(Q^*,d^*)$ and there is an open orbit in $\Rep(Q^*,d^*)$ whose complement $D^*$ is a hypersurface defined by a Saito determinant as in \eqref{27}. 
Then also $Z^*\subset\Sing(D^*)$ by reasons of invariance and dimension.
From Theorem~\ref{25}, one can then conclude that also condition~$\eqref{12c}$ in Theorem~\ref{12} is preserved under $(Q,d)\leftrightarrow(Q^*,d^*)$.
Thus, $D^*$ is a linear free divisor by Theorem~\ref{12}.
This proves Theorem~\ref{31}.

From Corollary~\ref{36} and Theorems \ref{38} and \ref{31} we obtain the following

\begin{cor}\label{37}
Each quiver defining a linear free divisor arises from an oriented bipartite graph by a sequence of reflection functors and adding arrows $k\to l$ with $d_k=1=d_l$.
In other words, one can make $m=1$ in Corollary~\ref{36} by such a sequence of operations.
Moreover, we have $\deg(f_{d,M})=\langle\pi^1(d),\pi^1(m)\rangle$ for $M\in M(Q,d)^\perp$ and $\deg(f_{M,d})=\langle\pi^0(d),\pi^0(m)\rangle$ for $M\in{^\perp M(Q,d)}$, where $m=\ddim(M)$ and $\pi^k\colon\NN^{\QQ_0}\to\NN^{Q_0^k}$ are the projections.
\end{cor}

Combining Proposition~\ref{97} and Corollary~\ref{37} proves the following

\begin{thm}\label{39}
Let $Q$ be a quiver with a real Schur root $d\in\NN^{Q_0}$ and let $m_1,\dots,m_r$ be the dimension vectors of the simple objects in $M(Q,d)^\perp$.
Assume that $Q$ is bipartite.
Then $D$ is a linear free divisor if and only if $\sum_{i=1}^r\langle\pi^1(d),\pi^1(m_i)\rangle=\dim\Rep(Q,d)$.
There is an analogous statement using $^\perp M(Q,d)$.
\end{thm}

\begin{arxiv}
\subsection{Examples}\label{84}

We shall give examples of linear free divisors in the tame case. 
The property of weak quasihomogeneity, and hence the local logarithmic comparison theorem (cf.~Theorem \ref{64}), holds at non regular points in $D$  by Theorem~\ref{34}. 
We shall focus on regular examples where the final and most sophisticated part of the proof of Theorem~\ref{50} is needed. 
We obtain in this way examples of linear free divisors for which the local comparison theorem was previously unknown even implicitly. 
If we eliminate the easy cases where the dimension vectors involve a terminal arrow with dimensions $(1,1)$ the interesting examples are supported by $\tilde E_7$ and $\tilde E_8$. 
We illustrate this in both cases, limiting ourself to the two highest regular dimension vectors and to the orientations of these quivers with arrows directed towards the triple node.     
We refer to \cite[Lec.~4]{CB92} for the list $\tilde A_n,\tilde D_m,\tilde A_n,\tilde E_6,\tilde E_7,\tilde E_8$ of tame quivers and the isotropic dimension vector $\delta$ of each one. 
According to the description in \S\ref{67}, for any dimension vector $d$ such that $d<\delta $, the number of isomorphism classes of representations is finite, and hence so is the number of orbits in $\Rep(Q,d)$. 
By Proposition~\ref{53} this covers the cases of all the regular linear free divisors.
 
In the following statement we study the effect of reflection functors in the regular case. 
We consider all the tame quivers with a fixed non oriented underlying graph $|Q|$ and we allow non sincere roots. 
Regularity is a property related to $Q$ and to $d\in \NN^{Q_0}$ and \emph{not} to the quiver obtained by deleting the vertices with $d_i=0$. 
Regularity is now preserved by any reflection functor relative to $|Q|$ because for any such reflection $d\to d^*$, we have $\delta^*=\delta$ and $\ideal{d,e}_Q=\ideal{d^*,e^*}_Q$, and by regularity $d^*\in \NN^{Q_0}\setminus\{0\}$ if $d\in \NN^{Q_0}\setminus\{0\}$. 
Notice also that these considerations are valid even if we include reflections centered at vertices $k$ such that $d_k=0$.
  
\begin{prp}
\begin{asparaenum}

\item Let $Q$ be a tame quiver and let $d$ be a dimension vector with $d<\delta$, and such that $M(Q,d)$ is indecomposable and regular. 
Then the discriminant in $\Rep(Q,d)$ is a linear free divisor. 

\item Let us consider the set of regular linear free divisors associated with  quivers $Q'$ having the same underlying graph $|Q|$ as $Q$ and dimension vector $d< \delta$. 
This set is finite and stable by reflection functors. 
Given any fixed orientation $Q$ of $|Q|$, any regular linear free divisor in $\Rep(Q',d)$ can be transformed by a sequence of reflections into a linear free divisor for $Q$. 

\item In the case of $\tilde E_7$ and $\tilde E_8$, and with a fixed orientation of arrows there are, respectively, $20$ and $28$ distinct regular indecomposable dimension vectors, or, equivalently, distinct regular linear free divisors, and among them respectively $6$ and $7$ classes, up to reflections. 

\end{asparaenum}
\end{prp}

\begin{proof}
\begin{asparaenum}

\item For the first statement we just use Lemma~\ref{82} which may be applied because the  finiteness of the number of isomorphism classes of representations of dimension strictly less than $\delta$ implies that each irreducible component $D_i$ of $D$ contains an open orbit.

\item The second statement is an easy combinatorial argument left to the reader.
 
\item For the last statement we use the fact (\cite{BB76}) that the Auslander--Reiten functor $\tau$ is the same as the Coxeter functor obtained as the product of one reflection at each vertex of $Q$. 
Therefore the classification that we want here is just the classification under the action of $\tau$. 
There are three indecomposable regular-simple dimension vectors $\alpha,\beta,\gamma$ which are of period $4,3,2$ and $5,3,2$, for $\tilde E_7$ and $\tilde E_8$ respectively. 
From this we easily get the stated number of dimension vectors, and the following list of reflection classes:
\begin{align*}
\text{For $\tilde E_7$: } &\alpha,\ \alpha+\tau\alpha,\ \alpha+\tau\alpha+\tau^2\alpha,\ \beta,\ \beta+\tau \beta,\ \gamma;\\
\text{for $\tilde E_8$: } & \alpha,\ \alpha+\tau\alpha,\ \alpha+\tau\alpha+\tau^2\alpha,\ \alpha+\tau\alpha+\tau^2\alpha+\tau^3\alpha,\ \beta,\ \beta+\tau \beta,\ \gamma.
\end{align*}
\end{asparaenum}
\end{proof}

In the case of $\tilde D_n$ or $\tilde E_6$ all the examples constructed in a similar way either correspond to a non sincere root or contain a terminal arrow with dimensions $(1,1)$ as in Proposition~\ref{24}.\eqref{24c}. 
In both cases the divisors are weakly quasihomogeneous as a direct consequence of the Dynkin case. 

In the case of $\tilde E_7$ and of $\tilde E_8$, with the orientation taking the central node as a unique sink, we find, respectively, one regular indecomposable representation for $\tilde E_7$ and four for $\tilde E_8$ which are not of the above type. 
For these, Theorem \ref{50} is needed to prove weak quasihomogeneity.

Let us conclude by two significant examples of such dimension vectors. 
In the case of $\tilde E_8$ we give the regular divisor of highest possible dimension. 
The dimension vector $d_{\max}$ shown below is equal to $\tau\alpha+\tau^2\alpha+\tau^3\alpha+\tau^4\alpha$, where $\alpha$ is the non sincere (regular) root with $q(\alpha)=4$ minimal, of period $5$. 
\[
\xymat@R=.2in@C=.3in{&&\KK^2\ar[d]^C\\
\KK^2\ar[r]^A&\KK^3\ar[r]^B&\KK^5&\KK^4\ar[l]_D&\KK^4\ar[l]_E&\KK^3\ar[l]_F&\KK^2\ar[l]_G&\KK\ar[l]_H}
\]
and $\dim\Rep(Q,d_{\max})=87$.
An equation for this divisor is 
\begin{align*}
\det(B)\det(B\vert C)\det(B\vert DEFG)\det(BA\vert&C\vert DEFGH)\det(BA\vert DEF) \det(C\vert DEF)\\
\det(B\vert DEFGH\vert\img(D)\cap\img(C)&)\det(C\vert DEFG\vert\img(D)\cap\img(BA))=0
\end{align*}
whose eight factors are of respective degrees $4$, $5$, $11$, $11$, $13$, $11$, $14$, and $18$. 
Here we write $\img(D)\cap\img(C)$ for two column vectors whose entries are polynomials of bidegree $(4,2)$ in the entries of $D$ and $C$, and which generate $\img(D)\cap\img(C)$ at a generic point of $\Rep(Q,d)$. 

In the case of $\tilde E_7$ the divisor is of dimension type $d=\tau\beta+\tau^2\beta$ for an appropriate simple regular root $\beta$ of period $3$. 
In this case $\dim\Rep(Q,d)=27$. 
\[
\xymat@R=.3in@C=.4in{
&&& \KK \ar[d]^G &&&\\
\KK \ar[r]^A & \KK^2 \ar[r]^B & \KK^2 \ar[r]^C & \KK^3 & \KK^2 \ar[l]_D & \KK^2 \ar[l]_E & \KK \ar[l]_F}
\]
An equation for this divisor is 
\[
\det(B)\det(E)\det(C\vert G)\det(D\vert G)\det(CBA\vert DEF\vert G)\det(CBA\vert D)\det(DEF\vert C)=0
\]
whose seven factors are of degrees $2$, $2$, $3$, $3$, $5$, $5$, and $7$.

\end{arxiv}

\bibliographystyle{amsalpha}
\bibliography{qlfd}

\def\cprime{$'$}
\providecommand{\bysame}{\leavevmode\hbox to3em{\hrulefill}\thinspace}
\providecommand{\MR}{\relax\ifhmode\unskip\space\fi MR }
\providecommand{\MRhref}[2]{%
  \href{http://www.ams.org/mathscinet-getitem?mr=#1}{#2}
}
\providecommand{\href}[2]{#2}
\begin{thebibliography}{CJGVHHU07}

\bibitem[ARO97]{ARS95}
Maurice Auslander, Idun Reiten, and Smal{\o}Sverre O., \emph{Representation
  theory of {A}rtin algebras}, Cambridge Studies in Advanced Mathematics,
  vol.~36, Cambridge University Press, Cambridge, 1997, Corrected reprint of
  the 1995 original. \MR{MR1476671 (98e:16011)}

\bibitem[BB76]{BB76}
Sheila Brenner and M.~C.~R. Butler, \emph{The equivalence of certain functors
  occurring in the representation theory of {A}rtin algebras and species}, J.
  London Math. Soc. (2) \textbf{14} (1976), no.~1, 183--187. \MR{0442031 (56
  \#420)}

\bibitem[BGP73]{BGP73}
I.~N. Bern{\v{s}}te{\u\i}n, I.~M. Gel{\cprime}fand, and V.~A. Ponomarev,
  \emph{Coxeter functors, and {G}abriel's theorem}, Uspehi Mat. Nauk
  \textbf{28} (1973), no.~2(170), 19--33. \MR{MR0393065 (52 \#13876)}

\bibitem[BM06]{BM06}
Ragnar-Olaf Buchweitz and David Mond, \emph{Linear free divisors and quiver
  representations}, Singularities and computer algebra, London Math. Soc.
  Lecture Note Ser., vol. 324, Cambridge Univ. Press, Cambridge, 2006,
  pp.~41--77. \MR{MR2228227 (2007d:16028)}

\bibitem[Bri07a]{Bri07a}
Michel Brion, \emph{Log homogeneous varieties}, Actas del XVI Coloquio
  Latinoamericano de \'Algebra (Madrid) (W.~Walter Ferrer~Santos, González
  Springer, Rittatore G., A., and A.~Solotar, eds.), Biblioteca de la Revista
  Matem\'atica Iberoamericana, 2007.

\bibitem[Bri07b]{Bri07b}
\bysame, \emph{Some remarks on linear free divisors}, communicated to the
  authors of \cite{GMNS09}, 2007.

\bibitem[Bri08]{Bri08}
\bysame, \emph{Representations of quivers},
  \url{http://www-fourier.ujf-grenoble.fr/IMG/pdf/notes_quivers.pdf}.

\bibitem[CB92]{CB92}
William Crawley-Boevey, \emph{Lectures on representations of quivers},
  \url{http://www.amsta.leeds.ac.uk/~pmtwc/quivlecs.pdf}.

\bibitem[CJGVHHU07]{CGHU07}
Francisco-Jes{\'u}s Castro-Jim{\'e}nez, Jes{\'u}s Gago-Vargas,
  Mar{\'{\i}}a-Isabel Hartillo-Hermoso, and Jos{\'e}-Mar{\'{\i}}a Ucha, \emph{A
  vanishing theorem for a class of logarithmic {$\mathcal{D}$}-modules},
  Proceedings of the {XVI}th {L}atin {A}merican {A}lgebra {C}olloquium
  ({S}panish), Bibl. Rev. Mat. Iberoamericana, Rev. Mat. Iberoamericana,
  Madrid, 2007, pp.~171--184. \MR{MR2500358}

\bibitem[dGMS09]{GMS09}
Ignacio de~Gregorio, David Mond, and Christian Sevenheck, \emph{Linear free
  divisors and {F}robenius manifolds}, Compos. Math. \textbf{145} (2009),
  no.~5, 1305--1350. \MR{MR2551998}

\bibitem[Dol03]{Dol03}
Igor Dolgachev, \emph{Lectures on invariant theory}, London Mathematical
  Society Lecture Note Series, vol. 296, Cambridge University Press, Cambridge,
  2003. \MR{MR2004511 (2004g:14051)}

\bibitem[FP09]{FP09}
L.M. Feh\'er and Zs. Patakfalvi, \emph{The incidence class and the hierarchy of
  orbits}, Central European J. Math. \textbf{7} (2009), no.~3, 429--441.

\bibitem[Gab72]{Gab72}
Peter Gabriel, \emph{Unzerlegbare {D}arstellungen. {I}}, Manuscripta Math.
  \textbf{6} (1972), 71--103; correction, ibid. 6 (1972), 309. \MR{MR0332887
  (48 \#11212)}

\bibitem[Gab80]{Gab80}
\bysame, \emph{Auslander-{R}eiten sequences and representation-finite
  algebras}, Representation theory, {I} ({P}roc. {W}orkshop, {C}arleton
  {U}niv., {O}ttawa, {O}nt., 1979), Lecture Notes in Math., vol. 831, Springer,
  Berlin, 1980, pp.~1--71. \MR{MR607140 (82i:16030)}

\bibitem[GMNRS09]{GMNS09}
Michel Granger, David Mond, Alicia Nieto-Reyes, and Mathias Schulze,
  \emph{Linear free divisors and the global logarithmic comparison theorem},
  Ann. Inst. Fourier (Grenoble) \textbf{59} (2009), no.~2, 811--850.
  \MR{MR2521436}

\bibitem[Gro66]{Gro66}
A.~Grothendieck, \emph{On the de {R}ham cohomology of algebraic varieties},
  Inst. Hautes \'Etudes Sci. Publ. Math. (1966), no.~29, 95--103. \MR{MR0199194
  (33 \#7343)}

\bibitem[GS06]{GS06}
Michel Granger and Mathias Schulze, \emph{On the formal structure of
  logarithmic vector fields}, Compos. Math. \textbf{142} (2006), no.~3,
  765--778. \MR{MR2231201}

\bibitem[KR86]{KR86}
H.~Kraft and Ch. Riedtmann, \emph{Geometry of representations of quivers},
  Representations of algebras ({D}urham, 1985), London Math. Soc. Lecture Note
  Ser., vol. 116, Cambridge Univ. Press, Cambridge, 1986, pp.~109--145.
  \MR{MR897322 (88k:16028)}

\bibitem[Kra84]{Kra84}
Hanspeter Kraft, \emph{Geometrische {M}ethoden in der {I}nvariantentheorie},
  Aspects of Mathematics, D1, Friedr. Vieweg \& Sohn, Braunschweig, 1984.
  \MR{MR768181 (86j:14006)}

\bibitem[NM08]{Nar08}
L.~Narv{\'a}ez~Macarro, \emph{Linearity conditions on the {J}acobian ideal and
  logarithmic-meromorphic comparison for free divisors}, Singularities {I},
  Contemp. Math., vol. 474, Amer. Math. Soc., Providence, RI, 2008,
  pp.~245--269. \MR{MR2454351}

\bibitem[Sai80]{Sai80}
Kyoji Saito, \emph{Theory of logarithmic differential forms and logarithmic
  vector fields}, J. Fac. Sci. Univ. Tokyo Sect. IA Math. \textbf{27} (1980),
  no.~2, 265--291. \MR{MR586450 (83h:32023)}

\bibitem[Sch91]{Sch91}
Aidan Schofield, \emph{Semi-invariants of quivers}, J. London Math. Soc. (2)
  \textbf{43} (1991), no.~3, 385--395. \MR{MR1113382 (92g:16019)}

\bibitem[Sch07]{Sch07}
Mathias Schulze, \emph{Maximal multihomogeneity of algebraic hypersurface
  singularities}, Manuscripta Math. \textbf{123} (2007), no.~4, 373--379.
  \MR{MR2320737}

\bibitem[Sev09]{Sev09}
Christian Sevenheck, \emph{{Bernstein} polynomials and spectral numbers for
  linear free divisors}, arXiv \textbf{math} (2009), no.~0905.0971, To appear
  in Ann. Inst. Fourier.

\bibitem[SK77]{SK77}
M.~Sato and T.~Kimura, \emph{A classification of irreducible prehomogeneous
  vector spaces and their relative invariants}, Nagoya Math. J. \textbf{65}
  (1977), 1--155. \MR{MR0430336 (55 \#3341)}

\bibitem[Wey97]{Wey97}
Hermann Weyl, \emph{The classical groups}, Princeton Landmarks in Mathematics,
  Princeton University Press, Princeton, NJ, 1997, Their invariants and
  representations, Fifteenth printing, Princeton Paperbacks. \MR{MR1488158
  (98k:01049)}

\end{thebibliography}

\end{document}